\pdfoutput=1
\RequirePackage{ifpdf}
\ifpdf % We~are running pdfTeX in pdf mode
\documentclass[pdftex]{sigma}
\else
\documentclass{sigma}
\fi

\numberwithin{equation}{section}

\newtheorem{Theorem}{Theorem}[section]
\newtheorem*{Theorem*}{Theorem}

\newtheorem{Lemma}[Theorem]{Lemma}
\newtheorem{Proposition}[Theorem]{Proposition}

\newtheorem{theoremAlph}{Theorem}
\newtheorem{corollaryAlph}[theoremAlph]{Corollary}
\newtheorem{Question}[Theorem]{Question}

\theoremstyle{definition}
\newtheorem{Definition}[Theorem]{Definition}

\newtheorem{Example}[Theorem]{Example}
\newtheorem{Remark}[Theorem]{Remark}

\numberwithin{table}{section}

\usepackage{tikz}
\usetikzlibrary{decorations,shapes,decorations.markings,decorations.pathreplacing}
\tikzset{
	OpenCirc/.style={
		label={[inner sep=0,minimum size=10mm, circle, draw, densely dotted]center:{}}
	},
	closed/.style = {decoration = {markings, mark = at position 0.5 with { \node[transform shape, xscale = .8, yscale=.4] {/}; } }, postaction = {decorate} },
	UpperBar/.style={decoration = {markings, mark = at position 0.25 with{ \arrow[line width=1pt]{|}
		} }, postaction={decorate} },
	LowerBar/.style={decoration = {markings, mark = at position 0.75 with { \arrow[line width=1pt]{|} } }, postaction={decorate} },
	kpm/.style={decoration = {markings, mark = at position 0.75 with {\node[draw=none, below =-1cm ]{$(k^+(u),k^-(u))$}} }, postaction={decorate} },
	10/.style={decoration = {markings, mark = at position 0.75 with { $(1,0)$ } }, postaction={decorate} },
	27/.style={decoration = {markings, mark = at position 0.75 with { $(2,7)$ } }, postaction={decorate} },
	Bullet/.style={circle, fill, draw, inner sep=0, minimum size=4pt}
}

\usepackage{longtable}
\usepackage{supertabular}

\usepackage{enumerate}
\usepackage{paracol}
\usepackage{multicol}
\usepackage[labelfont=bf,labelsep=period]{caption}
\usepackage{collcell}
\usepackage{array}
\newcolumntype{H}{@{}>{\collectcell\HideCell}c<{\endcollectcell}@{}}
\newcommand{\HideCell}[1]{}

\newcommand{\N}{\mathbb{N}}
\newcommand{\Z}{\mathbb{Z}}
\newcommand{\C}{\mathbb{C}}
\newcommand{\Q}{\mathbb{Q}}
\newcommand{\bigslant}[2]{{\raisebox{.2em}{$#1$}\left/\raisebox{-.2em}{$#2$}\right.}}
\newcommand{\ttimes}{\mathrel{\widetilde{\times} }}

\begin{document}

\allowdisplaybreaks

\newcommand{\arXivNumber}{2603.29838}

\renewcommand{\PaperNumber}{094}

\FirstPageHeading

\ShortArticleName{Free Circle Actions and Positive Ricci Curvature on Manifolds}

\ArticleName{Free Circle Actions and Positive Ricci Curvature\\ on Manifolds with the Cohomology Ring of $\boldsymbol{S^2\times S^5}$}

\Author{Philipp REISER}

\AuthorNameForHeading{P.~Reiser}

\Address{Karlsruhe, Germany}
\Email{\mail{philipp.reiser.math@mailbox.org}}
\URLaddress{\url{https://philipp-reiser.github.io/}}

\ArticleDates{Received April 21, 2026, in final form September 15, 2026; Published online October 02, 2026}

\Abstract{We classify which of the 672 oriented diffeomorphism types of closed, simply-connected spin 7-manifolds with the cohomology ring of $S^2\times S^5$ admit a free circle action. In~addition, we show that whenever such an action exists, there exist infinitely many pairwise non-equivalent free circle actions. Finally, in almost all cases where such an action exists, we construct invariant Riemannian metrics of positive Ricci curvature.}


\Classification{57S15; 53C20; 57R22; 55R25}

\section{Introduction and main results}

In this article, we study topological and geometric properties of free circle actions on closed manifolds. More precisely, the first question we consider is the following.
\begin{Question}\label{Q:free_circ_ex}
	Given a closed, simply-connected manifold $M$, does $M$ admit a free circle action?
\end{Question}
Here and throughout the article all manifolds and actions on manifolds are assumed to be smooth.

By a result of Kobayashi \cite{Ko58}, a closed manifold that admits a free circle action has vanishing Euler characteristic. Thus, the $3$-sphere $S^3$ is the only manifold in dimensions at most $4$ for which the answer to Question \ref{Q:free_circ_ex} is affirmative. In dimensions $5$ and $6$, Question \ref{Q:free_circ_ex} has been answered by Duan--Liang \cite{DL05} and by Goldstein--Lininger \cite{GL71} and Duan \cite{Du26}, respectively.

In dimension $7$, Montgomery--Yang \cite{MY68} showed that precisely $12$ out of the $28$ oriented diffeomorphism types of homotopy $7$-spheres admit a free circle action. Closed, $2$-connected $7$-manifolds that admit free circle actions were classified by Jiang \cite{Ji14} and manifolds homeomorphic to a connected sum of copies of $\bigl(S^2\times S^5\bigr)$ and $\bigl(S^3\times S^4\bigr)$ were considered by Xu \cite{Xu25}. A special case of the latter, which will be relevant for this article, is the following.
\begin{Theorem}[{\cite[Proposition 1]{Xu25}}]\label{T:S2xS5_hom}
	All $28$ oriented diffeomorphism types of closed $7$-manifolds that are homeomorphic to $S^2\times S^5$ admit a free circle action.
\end{Theorem}

Further results related to Question \ref{Q:free_circ_ex} are given by Hsiang \cite{Hs66}, Schultz \cite{Sc71} and Bauer--Quigley \cite{BQ26} for homotopy spheres, Jiang--Su \cite{JS25} for highly-connected manifolds and Galaz-García and the author \cite{GR25} for connected sums of products of spheres.

In light of Theorem~\ref{T:S2xS5_hom}, we will consider Question \ref{Q:free_circ_ex} for the following more general class of manifolds.	
\begin{Definition}
	A \emph{spin cohomology $S^2\times S^5$} is a smooth, closed, simply-connected spin $7$-mani\-fold whose cohomology ring is isomorphic to that of $S^2\times S^5$.
\end{Definition}

We will see in Section \ref{S:class} below that there are precisely $672$ oriented diffeomorphism types of spin cohomology $S^2\times S^5$s, which fall into $24$ oriented homeomorphism types (each containing $28$ oriented diffeomorphism types), and $18$ oriented homotopy types ($6$ of which contain $2$ oriented homeomorphism types, and $12$ contain a single oriented homeomorphism type).

Our first main result is given as follows.
\begin{theoremAlph}\label{T:free_circle}\qquad
	\begin{enumerate}\itemsep=0pt
		\item[$(1)$] Among the $672$ oriented diffeomorphism types of spin cohomology $S^2\times S^5$s, there are precisely $462$ that admit a free circle action. These cover $18$ out of the $24$ oriented homeomorphism types and $12$ out of the $18$ oriented homotopy types of spin cohomology $S^2\times S^5$s.
		\item[$(2)$] Whenever a spin cohomology $S^2\times S^5$ admits a free circle action, it admits infinitely many pairwise non-equivalent free circle actions.
	\end{enumerate}
\end{theoremAlph}

We will give a precise classification of the manifolds in (1) of Theorem~\ref{T:free_circle} in terms of their Kreck--Stolz $s$-invariants in Theorem~\ref{T:free_circle_s} below. We note that part (2) of Theorem~\ref{T:free_circle} appears to be new even in the simplest case of $S^2\times S^5$.

\begin{Remark}
	Since the existence of a free circle action does not depend on the choice of orientation, it is not necessary to restrict to diffeomorphisms, homeomorphisms and homotopy equivalences that are orientation-preserving. In this case, it follows from Remark~\ref{R:class_unor} and Theorem~\ref{T:free_circle_s} below that precisely $235$ out of $340$ (unoriented) diffeomorphism types of spin cohomology~${S^2\times S^5}$s admit a free circle action. However, we will only consider the oriented category since in this case it will be more convenient to analyse the $s$-invariants.
\end{Remark}

The second question we consider in this article is the following.
\begin{Question}\label{Q:free_circ_Ric}
	Given a closed, simply-connected manifold that admits a free circle action, does it admit a Riemannian metric of positive Ricci curvature that is invariant under this action?
\end{Question}

For spin cohomology $S^2\times S^5$s, Question \ref{Q:free_circ_Ric} is especially of interest since it was shown by Wang~\cite{Wa22} that \emph{all} spin cohomology $S^2\times S^5$s admit a Riemannian metric of positive Ricci curvature.

For scalar curvature, B\'erard-Bergery \cite{BB83} showed that a closed manifold with a free circle action admits an invariant Riemannian metric of positive scalar curvature if and only if the quotient space of the action admits a Riemannian metric of positive scalar curvature. It is open whether an analogous equivalence holds for positive Ricci curvature, although one half of it has been established by Gilkey--Park--Tuschmann \cite{GPT98}, namely if the manifold has finite fundamental group and the quotient space admits a Riemannian metric of positive Ricci curvature, then there exists an invariant Riemannian metric of positive Ricci curvature. We will use this fact, in combination with the existence of Riemannian metrics of positive Ricci curvature on certain $6$-manifolds established by the author in \cite{Re24b} to prove the following.	
\begin{theoremAlph}\label{T:Ric>0}
	There exist $441$ oriented diffeomorphism types of spin cohomology $S^2\times S^5$s that admit infinitely many pairwise non-equivalent free circle actions and for each of these actions an invariant Riemannian metric of positive Ricci curvature. These include all $28$ oriented diffeomorphism types of spin cohomology $S^2\times S^5$s that are homeomorphic to $S^2\times S^5$.
\end{theoremAlph}

Note that in Theorem~\ref{T:Ric>0} we do not claim that for these manifolds \emph{any} free circle action admits an invariant Riemannian metric of positive Ricci curvature.

We will again give a precise characterisation of the manifolds in Theorem~\ref{T:Ric>0} in terms of their $s$-invariants, see Theorem~\ref{T:Ric>0_s} below. These include all $420$ oriented diffeomorphism types that admit a free circle action with spin orbit space. The remaining $21$ oriented diffeomorphism types that admit a free circle action by Theorem~\ref{T:free_circle} but are not covered by Theorem~\ref{T:Ric>0} therefore only have free circle actions with non-spin orbit space. It remains open whether these manifolds admit invariant Riemannian metrics of positive Ricci curvature, cf.\ Remark \ref{R:Ric>0_missing} below.
\begin{Question}\label{Q:Ric>0_missing}
	Which of the $21$ oriented diffeomorphism types of spin cohomology $S^2\times S^5$s that admit a free circle action by Theorem~\ref{T:free_circle} but are not covered by Theorem~\ref{T:Ric>0} admit an invariant Riemannian metric of positive Ricci curvature?
\end{Question}

We will see in Proposition \ref{P:pi4} below that the $672$ oriented diffeomorphism types of spin cohomology $S^2\times S^5$s can be divided into two subfamilies, each one containing $336$ oriented diffeomorphism types: Those with $\pi_4(M)\cong\Z/2$ \big(such as $M=S^2\times S^5$\big) and those with~${\pi_4(M)=0}$. A direct consequence of Proposition~\ref{P:pi4} and Theorems \ref{T:free_circle_s} and \ref{T:Ric>0_s} is as follows.
\begin{corollaryAlph}
	Let $M$ be a spin cohomology $S^2\times S^5$ with $\pi_4(M)=0$. Then $M$ admits infinitely many pairwise non-equivalent free circle actions and for each of these actions an invariant Riemannian metric of positive Ricci curvature.
\end{corollaryAlph}

Finally, we will use the results of Theorems \ref{T:free_circle} and \ref{T:Ric>0} in combination with the suspension construction of \cite{Du26} and \cite{GR25} to prove the following.
\begin{theoremAlph}\label{T:conn_sums}
	Let $\Sigma$ be a homotopy $7$-sphere and let
	\[ M=\#_\ell\bigl(S^2\times S^5\bigr)\#_m\bigl(S^3\times S^4\bigr)\#\Sigma. \]
	If $\ell\geq 2$, or $\ell=1$ and $m$ is even, then $M$ admits infinitely many pairwise non-equivalent free circle actions and for each of these actions an invariant Riemannian metric of positive Ricci curvature.
\end{theoremAlph}
Again we do not claim that $M$ admits an invariant Riemannian metric of positive Ricci curvature for \emph{any} free circle action on $M$.

The existence of a free circle action on each of these manifolds in Theorem~\ref{T:conn_sums} has already been established by Xu \cite{Xu25}. Moreover, by the work of Wraith \cite{Wr97}, Burdick \cite[Corollary 1.3.5]{Bu19a}, \cite[Theorem B]{Bu19}, \cite[Theorem B]{Bu20} and the author \cite[Theorem C]{Re23}, they also admit Riemannian metrics of positive Ricci curvature. If $\Sigma$ is the standard sphere, invariant Riemannian metrics of positive Ricci curvature have been shown to exist in \cite[Corollary H]{GR25}. However, to the best of our knowledge, when $\Sigma$ is not the standard sphere, the existence of an invariant Riemannian metric of positive Ricci curvature, even under a single free circle action, is new in the literature.

This article is organised as follows. In Section \ref{S:prel}, we recall basic facts on free circle actions and on closed, simply-connected 6-manifolds. In Section \ref{S:class}, we classify spin cohomology $S^2\times S^5$s up to orientation-preserving diffeomorphism, homeomorphism and homotopy equivalence. In Section~\ref{S:circ_bund}, we prove Theorems~\ref{T:free_circle} and \ref{T:Ric>0} by first determining which $6$-manifolds can appear as the orbit space of a free circle action on a spin cohomology $S^2\times S^5$ following \cite{Xu25}, and then calculating the $s$-invariants of total spaces of principal circle bundles over these $6$-manifolds to analyse which spin cohomology $S^2\times S^5$s we obtain in this way. Since the formulae for the latter are highly complicated, we used a computer to explicitly show that each of the manifolds claimed in Theorem~\ref{T:free_circle} can be realised in this way, the results of which are summarised in Appendix \ref{S:Tables}. Finally, in Section \ref{S:further}, we prove Theorem~\ref{T:conn_sums} and further related results.

\section{Preliminaries}\label{S:prel}

Unless stated otherwise, all manifolds, bundles and maps between manifolds are assumed to be smooth, and we assume that all manifolds are connected. Further, if no coefficients are indicated, we consider \text{(co-)homology} with coefficients in $\Z$.

\subsection{Free circle actions}

In this subsection, we recall basic facts on free circle actions that we will use. We first recall the definition of equivalence of two actions.

\begin{Definition}
	Let $M$ be a manifold. Then two free circle actions $\rho_1,\rho_2\colon S^1\times M\to M$ are \emph{equivalent} if there is a diffeomorphism $\phi\colon M\to M$ such that
	\[ \phi(\rho_1(\lambda,x))=\rho_2(\lambda,\phi(x)) \]
	for all $(\lambda,x)\in S^1\times M$.
\end{Definition}

Free circle actions can equivalently be described through principal circle bundles as we will see in Proposition \ref{P:action=bdl} below. For that, recall that a principal circle bundle $M\xrightarrow{\pi}N$ is, up to isomorphism of principal bundles, uniquely determined by its Euler class $e(\pi)\in H^2(N)$. The latter corresponds, under the identification
\[H^2(N)\cong[N,K(\Z,2)]\cong \bigl[N,\mathrm{B}S^1\bigr],\]
to the classifying map $N\to \mathrm{B}S^1$ of the bundle $\pi$ (see, e.g., \cite[Sections 4.10--4.13]{Hu94} for more details on classifying spaces). As a consequence, any two principal circle bundles \smash{$M_1\xrightarrow{\pi_1}N_1$} and \smash{$M_2\xrightarrow{\pi_2}N_2$} are isomorphic if and only if there is a diffeomorphism $\phi\colon N_1\to N_2$ with ${\phi^*(e(\pi_2))=e(\pi_1)}$.

\begin{Proposition}\label{P:action=bdl}
	Let $M$ be a manifold of dimension $n$.
	\begin{enumerate}\itemsep=0pt
		\item[$(1)$] The following are equivalent:
		\begin{enumerate}[$(i)$]\itemsep=0pt
			\item $M$ admits a free circle action.
			\item $M$ is the total space of a principal circle bundle \smash{$M\xrightarrow{\pi} N$}, where $N$ is a manifold of dimension $(n-1)$.
		\end{enumerate}
		\item[$(2)$] Suppose $M$ admits two free circle actions and let \smash{$M\xrightarrow{\pi_1} N_1$} and \smash{$M\xrightarrow{\pi_2} N_2$} be the corresponding principal circle bundles as in $(ii)$. Then the two actions are equivalent if and only if there is a diffeomorphism $\phi\colon N_1\to N_2$ with $\phi^*(e(\pi_2))=e(\pi_1)$.
	\end{enumerate}
\end{Proposition}
\begin{proof}
	For (1), see, e.g., \cite[Corollary~VI.2.5]{Br72a}. Part~(2) follows from the fact that, by definition, two principal circle bundles are isomorphic if and only if the corresponding circle actions are equivalent.
\end{proof}

We will use the following result to endow manifolds with free circle actions with a Riemannian metric of positive Ricci curvature, which is due to Gilkey--Park--Tuschmann \cite{GPT98}, with a~preliminary version due to B\'erard-Bergery \cite{BB78}.
\begin{Theorem}[{\cite{BB78,GPT98}}]\label{T:Ric>0_bdl}
	Let $M$ be a closed manifold that admits a free circle action. Suppose that $M$ has finite fundamental group and the quotient space $N$ $($which is a closed manifold$)$ admits a Riemannian metric of positive Ricci curvature. Then $M$ admits a Riemannian metric of positive Ricci curvature that is invariant under the free circle action.
\end{Theorem}

Finally, the following lemma will be useful for the homotopy classification of spin cohomology~${S^2\times S^5}$s.
\begin{Lemma}\label{L:princ_circ_hom}
	Let $N$ be a manifold and let $M\to N$ be a principal circle bundle. Then $\pi_i(N)\cong\pi_i(M)$ for all $i\geq 3$.
\end{Lemma}
\begin{proof}
	This is a direct consequence of the long exact sequence of homotopy groups of the fibration $M\to N$ (see, e.g., \cite[Corollary 6.44]{DK01}), in combination with the fact that $\pi_i\bigl(S^1\bigr)=0$ for all $i\geq 2$.
\end{proof}

\subsection{Simply-connected 6-manifolds}

Closed, simply-connected $6$-manifolds were classified under additional assumptions by Wall \mbox{\cite{Wa66,Wa67}} (spin, torsion-free homology) and Jupp \cite{Ju73} (torsion-free homology) and in the general case by Zhubr \cite{Zu75,Zh00}. Only the case of torsion-free homology will be relevant in this article, so we will focus on Jupp's classification.

For that, let $N$ be a closed, simply-connected $6$-manifold with torsion-free homology. Then $H^2(N)$ is a free abelian group on which we have the symmetric trilinear form $\mu_N\colon H^2(N)\times H^2(N)\times H^2(N)\to \Z$ defined by
\[ \mu_N(x_1,x_2,x_3)=\langle x_1\smile x_2\smile x_3,[N]\rangle, \]
where we chose an arbitrary orientation on $N$.

Further, we have the linear form $p_1(N)\colon H^2(N)\to \Z$ defined by the first Pontryagin class
\[ p_1(N)(x)=\langle p_1(N)\smile x,[N]\rangle \]
and the second Stiefel--Whitney class $w_2(N)$ defines an element of $H^2(N;\Z/2)\cong H^2(N)\otimes\Z/2$. Finally, the third Betti number $b_3(N)$, which is always even, defines a non-negative integer \smash{\raisebox{-0.3pt}{$\frac{b_3(N)}{2}$}}.	The \emph{system of invariants assigned to $N$} is the $5$-tuple \smash{\raisebox{-0.3pt}{$\bigl(H^2(N),\mu_N,p_1(N),w_2(N),\frac{b_3(N)}{2}\bigr)$}}.

More abstractly, we call a $5$-tuple $(H,\mu,p,w,r)$ consisting of a finitely generated free abelian group $H$, a symmetric trilinear form $\mu\colon H\times H \times H\to \Z$, a linear form $p\colon H\to \Z$, an element $w\in H\otimes\Z/2$ and a non-negative integer $r$, a \emph{system of invariants}. Two systems of invariants $(H,\mu,p,w,r)$ and $(H',\mu',p',w',r')$ are \emph{equivalent} if $r=r'$ and there exists an isomorphism $\phi\colon H\to H'$ such that $\phi^*\mu'=\mu$, $\phi^* p'=p$ and $\phi(w)=w'$.

Note that the definition of the system of invariants $\bigl(H^2(N),\mu_N,p_1(N),w_2(N),\frac{b_3(N)}{2}\bigr)$ depends on a choice of orientation of $N$. Reversing the orientation of $N$ results in replacing $\mu_N$ and $p_1(N)$ by $-\mu_N$ and $-p_1(N)$, respectively. Since the systems of invariants
\[\Big(H^2(N),\mu_N,p_1(N),w_2(N),\tfrac{b_3(N)}{2}\Big)\qquad \text{and}\qquad \Big(H^2(N),-\mu_N,-p_1(N),w_2(N),\tfrac{b_3(N)}{2}\Big)\]
are equivalent via the isomorphism $-\mathrm{id}_{H^2(N)}$, the choice of orientation does not need to be part of the data.

\begin{Theorem}[{\cite{Ju73}}]\label{T:Jupp}
	Two closed, simply-connected $6$-manifolds $N_1$ and $N_2$ with torsion-free homology are diffeomorphic if and only if their systems of invariants are equivalent.
	
	Moreover, a system of invariants $(H,\mu,p,w,r)$ is realised by a closed, simply-connected $6$-manifold with torsion-free homology if and only if
	\begin{equation}\label{EQ:Jupp}
		\mu(W,W,W)\equiv p(W)\mod 48
	\end{equation}
	for all $W\in H$ with $W\equiv w\mod 2$.
\end{Theorem}
Note that, when $w=0$, by setting $W=2x$ for any $x\in H$, \eqref{EQ:Jupp} becomes
\begin{equation}\label{EQ:Wall}
	4\mu(x,x,x)\equiv p(x)\mod 24,
\end{equation}
which is precisely the condition identified by Wall \cite[Theorem 5]{Wa66} in the spin case.

The sphere $S^6$ is the unique manifold with trivial system of invariants. Further manifolds we will consider are as follows.

\begin{Lemma}\label{L:inv_examples}\qquad
	\begin{enumerate}\itemsep=0pt
		\item[$(1)$] The system of invariants of the product $S^2\times S^4$ is equivalent to $(\Z,0,0,0,0)$.
		\item[$(2)$] The system of invariants of the complex projective space $\C P^3$ is given by $\bigl(\Z a,\mu_{\C P^3},\allowbreak p_1\bigl(\C P^3\bigr),0,0\bigr)$, where $a\in H^2\bigl(\C P^3\bigr)\cong \Z$ is a generator, and $\mu_{\C P^3}$ and $p_1\bigl(\C P^3\bigr)$ are given by
$\mu_{\C P^3}(a,a,a)=1$, $p_1\bigl(\C P^3\bigr)(a)=4$.
		\item[$(3)$] If $N_1$ and $N_2$ are closed, simply-connected 6-manifolds with torsion-free homology, then system of invariants of the connected sum $N_1\# N_2$ is equivalent to
		\[ \left(H^2(N_1)\oplus H^2(N_2),\mu_{N_1}\oplus\mu_{N_2},p_1(N_1)\oplus p_1(N_2),(w_2(N_1),w_2(N_2)),\tfrac{b_3(N_1)+b_3(N_2)}{2}  \right). \]
	\end{enumerate}
\end{Lemma}
\begin{proof}
	This is well known, see, e.g., \cite[Example~15.6]{MS74} for the Pontryagin classes of $\C P^3$ and~\mbox{\cite[Lemma~4.1.4]{Re22}} for the statement on connected sums.
\end{proof}

We will need the following result for the proof of Theorem~\ref{T:conn_sums}.
\begin{Proposition}\label{P:splitting_diffeo}
	Let $N$, $N_1$ and $N_2$ be closed, simply-connected 6-manifolds with torsion-free homology such that $b_2(N)=1$ and either $\mu_N$ is non-trivial or $N=S^2\times S^4$. Suppose that $N_1\# N$ is diffeomorphic to $N_2\# N$. Then $N_1$ and $N_2$ are diffeomorphic.
\end{Proposition}
For the proof we need the following lemma. For that, given a closed, simply-connected 6-manifold $N$ with torsion-free homology, we define the kernel $\ker(\mu_N,p_1(N))\subseteq H^2(N)$ of the pair $(\mu_N,p_1(N))$ by
\[ \ker(\mu_N,p_1(N))=\bigl\{ x\in H^2(N)\mid \mu_{N}(x,y,z)=p_1(N)(x)=0\text{ for all }y,z\in H^2(N)\bigr\}. \]
\begin{Lemma}\label{L:ker_dir_sum}
	The kernel $\ker(\mu_N,p_1(N))$ is a direct summand inside $H^2(N)$.
\end{Lemma}
\begin{proof}
	It follows from the multilinearity of the form $\mu_N$ and the linearity of the form $p_1(N)$ that $K=\ker(\mu_N,p_1(N))$ is a~subspace of \smash{$H^2(N)$}. Further, by the elementary divisor theorem, there exists a basis $(u_1,\dots,u_k)$ of \smash{$H^2(N)$} and $\lambda_1,\dots,\lambda_\ell\in \N$ such that $(\lambda_1 u_1,\dots,\lambda_\ell u_\ell)$ is a~basis of~$K$. By definition, since $\lambda_i u_i\in K$, we have $u_i\in K$ and so $\lambda_i=1$. Hence, $K$ is a direct summand in~\smash{$H^2(N)$}.
\end{proof}

\begin{proof}[Proof of Proposition \ref{P:splitting_diffeo}]
	Set $H_1=H^2(N_1)$ and $H_2=H^2(N_2)$ and let $a\in H^2(N)\cong\Z$ be a generator. Then, by Lemma~\ref{L:inv_examples}, we can identify $H^2(N_i\# N)\cong H_i\oplus \Z a$ and the trilinear form $\mu_{N_i\# N}$ splits according to this direct sum. A diffeomorphism $N_1\# N\to N_2\# N$ induces an~isomorphism $\phi\colon H_1\oplus \Z a\to H_2\oplus \Z a$ such that $\phi^*\mu_{N_2\# N}=\mu_{N_1\# N}$. In particular, $H_1$ and~$H_2$ have the same rank. Further, since $b_3(N_i\# N)=b_3(N_i)+b_3(N)$, we have $b_3(N_1)=b_3(N_2)$.
	
	Now assume that $\mu_N$ is non-trivial. Let $(u_1,\dots,u_k)$ be a basis of $H_1$ and $(v_1,\dots,v_k)$ a basis of $H_2$. We write
	\[ \phi(u_i)=\sum_j \mu_{ij} v_j+\lambda_i a,\qquad \phi(a)=\sum_j\mu_j v_j+\lambda a \]
	with $\mu_{ij},\mu_j,\lambda_i,\lambda\in\Z$.
	
	In the following, we will identify a trilinear form $\mu\colon H\times H\times H\to\Z$ with a bilinear map $H\times H\to H^*=\mathrm{Hom}(H,\Z)$ via $(x,y)\mapsto (z\mapsto\mu(x,y,z))$ and we will write $xy\in H^*$ for the image of $(x,y)$ under this map. For example, we have $a^2=\mu_N(a,a,a)a^*\in (\Z a)^*$, where $a^*(a)=1$. We set $\alpha=\mu_N(a,a,a)$, which, by assumption, is non-zero.
	
	With this notation, we have $u_i a=0$, and hence
	\[ 0=\phi(u_i)\phi(a)=\sum_{j,\ell}\mu_{ij}\mu_\ell v_j v_\ell +\lambda_i\lambda a^2=\sum_{j,\ell}\mu_{ij}\mu_\ell v_j v_\ell +\lambda_i\lambda\alpha a^*.  \]
	
	Since $v_j v_\ell\in H_2^*$, it follows that either $\lambda_i=0$ for all $i$ or $\lambda=0$. We distinguish these two~cases:
%	\begin{itemize}\setlength{\leftskip}{0.44cm}\itemsep=0pt

\emph{Case 1}: $\lambda_i=0$ for all $i$. In this case, $\phi(u_i)\in H_2$ for all $i$. Hence, $\phi$ restricts to an isomorphism $\phi\colon H_1\to H_2$, which, by assumption, preserves $\mu$ and $p_1$. Thus, the systems of invariants $\bigl(H_1,\mu_{N_1},p_1(N_1),w_2(N_1),\tfrac{b_3(N_1)}{2}\bigr)$ and $\bigl(H_2,\mu_{N_2},p_1(N_2),w_2(N_2),\tfrac{b_3(N_2)}{2}\bigr)$ are equivalent via $\phi$ and so $N_1$ and $N_2$ are diffeomorphic.

\emph{Case 2}: $\lambda=0$. Since $(\phi(u_1),\dots,\phi(u_k),\phi(a))$ is a basis of $H_2\oplus\Z a$ and $\phi(a)$ has no $a$\nobreakdash-component, we can choose the basis $(u_1,\dots,u_k)$ such that $\lambda_1=1$ and $\lambda_i=0$ for all $i\in\{2,\dots,k\}$. In particular, $(\phi(u_2),\dots,\phi(u_k),\phi(a))$ is a basis of $H_2$. Hence, there are $\nu_2,\dots,\nu_k$, ${\nu\in\Z}$ with
		\[ \phi(u_1)-a=\nu_2\phi(u_2)+\dots+\nu_k\phi(u_k)+\nu\phi(a) \]
		and so
		\[ \phi(u_1-\nu_2 u_2 -\dots-\nu_k u_k )-a=\nu\phi(a). \]
		Thus, if we replace $u_1$ by $u_1-\nu_2 u_2 -\dots-\nu_k u_k$, which still results in a basis $(u_1,\dots,u_k)$ of $H_1$, we have
		\[ \phi(u_1)=a+\nu\phi(a). \]
Hence,
		\begin{align*}
        0&=\mu_{N_2\#N}(\phi(u_1),\phi(a),\phi(a))\\
        &=\mu_{N_2\#N}(a,\phi(a),\phi(a))+\nu\mu_{N_2\#N}(\phi(a),\phi(a),\phi(a))=\nu\alpha,
        \end{align*}
		so $\nu=0$ and $\phi(u_1)=a$. Since $\mu_{N_2\# N}$ splits according to the direct sum decomposition $H_2\oplus\Z a$, the decomposition $H_1\oplus \Z a=\Z u_1\oplus \mathrm{span}(u_2,\dots,u_k)\oplus \Z a$ is therefore a~splitting of $\mu_{N_1\# N}$.
		
		By defining the basis $v_1,\dots,v_k$ so that $\phi(a)=v_1$ and $\phi(u_i)=v_i$ for $i=2,\dots,k$ (recall that $(\phi(u_2),\dots,\phi(u_k),\phi(a))$ is a basis of $H_2$), we therefore obtain that the isomorphism
		\[ \phi\colon\ \Z u_1\oplus H_1'\oplus \Z a\to \Z a\oplus H_2'\oplus \Z v_1 \]
		splits according to this decomposition, where we have set $H_1'=\mathrm{span}(u_2,\dots,u_k)$ and $H_2'=\mathrm{span}(v_2,\dots,v_k)$. If $N_i'$ denotes the closed, simply-connected $6$-manifold with torsion-free homology with invariants \smash{$\bigl(H_i',\mu_{N_i}|_{H_i'},p_1(N_i)|_{H_i'},w_2(N_i)|_{H_i'},\tfrac{b_3(N_i)}{2}\bigr)$}, we therefore obtain that there is a~diffeomorphism
		\[ N\# N_1'\# N\to N\# N_2'\# N \]
		that preserves this connected sum decomposition. In particular, $N_1\cong N_1'\# N$ is diffeomorphic to $N_2\cong N_2'\# N$.
%	\end{itemize}
	
	Finally, assume that $N=S^2\times S^4$. We set $K_i=\ker(\mu_{N_i},p_1(N_i))$. Then, by Lemma~\ref{L:inv_examples},
	\[\ker(\mu_{N_i\# N},p_1(N_i\#N))=K_i\oplus\Z a.\]
	By Lemma~\ref{L:ker_dir_sum}, there exists a complement $H_i'$ of $K_i$ inside $H_i$. We choose the complement $H_i'$ so that $w_2(N_i)\in H_i'\mod 2$ whenever $w_2(N_i)\not\in K_i\mod 2$. Hence,
	\[ \phi\colon\ H_1'\oplus K_1\oplus\Z a\to H_2'\oplus K_2\oplus\Z a \]
	is an isomorphism that maps $K_1\oplus\Z a$ to $K_2\oplus\Z a$ since $\phi$ preserves $\mu_{N_i\# N}$ and $p_1(N_i\# N)$. Since~$\phi$ also preserves $w_2(N_i)$, we have $w_2(N_1)\in H_1'\mod 2$ if and only if $w_2(N_2)\in H_2'\mod 2$.
	
	If $\mathrm{pr}_{H_2'}$ denotes the projection to $H_2'$ along this splitting, the composition
	\[ H_1'\xrightarrow{\phi|_{H_1'}}H_2'\oplus K_2\oplus\Z a\xrightarrow{\mathrm{pr}_{H_2'}}H_2' \]
	is therefore an isomorphism that preserves $w_2(N_i)$ if $w_2(N_1)\in H_1'\mod 2$. We claim that it also preserves $\mu_{N_i}|_{H_i'}$ and $p_1(N_i)|_{H_i'}$. Indeed, if we write $\phi(x)=x'+x_K$ with $x'\in H_2'$ and $x_K\in K_2\oplus\Z a$ for $x\in H_1'$, we have for $x,y,z\in H_1'$,
	\begin{align*}
		\mu_{N_1\# N}(x,y,z)&=\mu_{N_2\# N}(\phi(x),\phi(y),\phi(z)
		=\mu_{N_2\# N}\bigl(x'+x_K,y'+y_K,z'+z_K\bigr)\\
		&=\mu_{N_2\# N}\bigl(x',y',z'\bigr)
		=\mu_{N_2\# N}\bigl(\mathrm{pr}_{H_2'}(\phi(x)),\mathrm{pr}_{H_2'}(\phi(y)),\mathrm{pr}_{H_2'}(\phi(z))\bigr)
	\end{align*}
	and
	\begin{align*}
		p_1(N_1\# N)(x)&=p_1(N_2\# N)(\phi(x))
        =p_1(N_2\# N)\bigl(x'+x_K\bigr)\\
        &=p_1(N_2\# N)\bigl(x'\bigr)
        =p_1(N_2\# N)\bigl(\mathrm{pr}_{H_2'}(\phi(x))\bigr).
	\end{align*}
Hence, $H_1'$ and $H_2'$ are isomorphic via an isomorphism that preserves $\mu_{N_i}|_{H_i'}$ and $p_1(N_i)|_{H_i'}$. Since both the trilinear and linear form are trivial on $K_i$ and $w_2(N_1)\in K_1$ if and only if $w_2(N_2)\in K_2$, this isomorphism extends to an isomorphism between $H_1=H_1'\oplus K_1$ and ${H_2=H_2'\oplus K_2}$ that preserves $\mu_{N_i}$, $p_1(N_i)$ and $w_2(N_i)$. It follows that the systems of invariants \smash{$\bigl(H_1,\mu_{N_1},p_1(N_1),w_2(N_1),\tfrac{b_3(N_1)}{2}\bigr)$} and \smash{$\bigl(H_2,\mu_{N_2},p_1(N_2),w_2(N_2),\tfrac{b_3(N_2)}{2}\bigr)$} are equivalent and~hence, by Theorem~\ref{T:Jupp}, $N_1$ and $N_2$ are diffeomorphic.
\end{proof}

\begin{Remark}
	The conclusion of Proposition \ref{P:splitting_diffeo} does not hold if we merely assume $b_2(N)=1$. For example, let $N$ denote the unique closed, simply-connected spin $6$-manifold with torsion-free homology and invariants $(\Z x,0,24x^*,0,0)$ (which exists by Theorem~\ref{T:Jupp} and \eqref{EQ:Wall}). Then, if $y\in H^2\bigl(S^2\times S^4\bigr)$ denotes a generator and $x_i\in H^2(N)$ is a generator of the $i$-th $N$-copy in $N\# N$, the map
	\[ y\mapsto x_1-x_2,\qquad x\mapsto x_2 \]
	induces a diffeomorphism between $\bigl(S^2\times S^4\bigr)\# N$ and $N\#N$, but $N$ is not diffeomorphic to $S^2\times S^4$ since $p_1(N)\neq 0$.
\end{Remark}

\section[Classification of spin cohomology S\^{}2 times S\^{}5s]{Classification of spin cohomology $\boldsymbol{S^2\times S^5}$s}\label{S:class}

In this section, we classify spin cohomology $S^2\times S^5$s up to orientation-preserving diffeomorphism, homeomorphism and homotopy equivalence. For that we will use Kreck--Stolz $s$-invariants. These were introduced and used by Kreck and Stolz in \cite{KS88,KS91,KS98} to analyse closed, simply-connected $7$-manifolds $M$ with $H^2(M)\cong \Z$, $H^3(M)=0$ and $H^4(M)$ generated by the square of a generator of $H^2(M)$. In particular, a spin cohomology $S^2\times S^5$ satisfies these conditions. The invariants, denoted by $s_1$, $s_2$ and $s_3$, are computed from characteristic numbers of a cobounding manifold and take values in $\Q/\Z$. The oriented diffeomorphism type of a manifold is then uniquely determined by $s_1$, $s_2$ and $s_3$, and the oriented homeomorphism type by $28s_1$, $s_2$ and~$s_3$ (see \cite[Section 3]{KS88}). We refer to \cite{KS88,KS91,KS98} and \cite[Section 2.2]{Xu25} for more details and basic properties.

In our case, it will be more convenient to consider the modified invariant $\tilde{s}_3=s_3-4s_2$ instead of~$s_3$. Since $s_3$ can be recovered from~$\tilde{s}_3$ and~$s_2$, there is no loss in information when replacing~$s_3$ by~$\tilde{s}_3$. We obtain the following classification result.
\begin{Theorem}\label{T:class}\qquad
	\begin{enumerate}\itemsep=0pt
		\item[$(1)$] For spin cohomology $S^2\times S^5$s, the invariants $s_1$, $s_2$, $\tilde{s}_3$ take the following values:
		\[  s_1\in \bigslant{\bigl(\tfrac{1}{28}\Z\bigr)}{\Z}, \qquad s_2\in \bigslant{\bigl(\tfrac{1}{12}\Z\bigr)}{\Z}, \qquad \tilde{s}_3\in \bigslant{\bigl(\tfrac{1}{2}\Z\bigr)}{\Z}. \]
		Moreover, all of these values can be realised.
		\item[$(2)$] Let $M$ and $M'$ be two spin cohomology $S^2\times S^5$s, each equipped with an orientation.
		\begin{enumerate}\itemsep=0pt
			\item[$(a)$] $M$ and $M'$ are orientation-preserving diffeomorphic if and only if $s_1(M)=s_1(M')$, $s_2(M)=s_2(M')$ and $\tilde{s}_3(M)=\tilde{s}_3(M')$. In particular, there are $28\cdot 12\cdot 2=672$ oriented diffeomorphism types of spin cohomology $S^2\times S^5$s.
			\item[$(b)$] $M$ and $M'$ are orientation-preserving homeomorphic if and only if $s_2(M)=s_2(M')$ and $\tilde{s}_3(M)=\tilde{s}_3(M')$. In particular, there are $12\cdot 2=24$ oriented homeomorphism types of spin cohomology $S^2\times S^5$s.
			\item[$(c)$] $M$ and $M'$ are orientation-preserving homotopy equivalent if and only if $s_2(M)=s_2(M')$ and $\tilde{s}_3(M)=\tilde{s}_3(M')=0$, or $2s_2(M)=2s_2(M')$ and $\tilde{s}_3(M)=\tilde{s}_3(M')=\frac{1}{2}$. In particular, there are $12+6=18$ oriented homotopy types of spin cohomology~${S^2\times S^5}$s.
		\end{enumerate}
	\end{enumerate}
\end{Theorem}
While the diffeomorphism and homeomorphism classification was essentially shown by Wang in~\cite{Wa22}, to the best of our knowledge, Theorem~\ref{T:class} provides the first complete homotopy classification of spin cohomology $S^2\times S^5$s. It makes use of work of Kruggel \cite{Kr97,Kr98} on the homotopy type of certain simply-connected 7-manifolds.

\begin{Remark}\label{R:class_unor}
Reversing the orientation has the effect of multiplying the invariants $s_1$, $s_2$ and~$\tilde{s}_3$ by $(-1)$. Hence, we can derive the unoriented diffeomorphism, homeomorphism and homotopy classification of spin cohomology $S^2\times S^5$s from Theorem~\ref{T:class}. In particular, there are precisely $8$ diffeomorphism types, $4$ homeomorphism types and $4$ homotopy types that admit an~orientation-reversing diffeomorphism, homeomorphism and homotopy equivalence, respectively. Consequently, there exist $340$ diffeomorphism types, $14$ homeomorphism types and $11$~homotopy types of spin cohomology $S^2\times S^5$s.
\end{Remark}

\begin{Example}\label{EX:spin_cohom_S2xS5}\qquad
	\begin{enumerate}\itemsep=0pt
		\item[$(1)$] Clearly, the manifold $S^2\times S^5$ is a spin cohomology $S^2\times S^5$. Its $s$-invariants are given by
		\[ s_1\bigl(S^2\times S^5\bigr)=s_2\bigl(S^2\times S^5\bigr)=\tilde{s}_3\bigl(S^2\times S^5\bigr)=0, \]
		see, e.g., \cite[Lemma~4]{Xu25}.
		\item[$(2)$] Since $\pi_4(\mathrm{SO}(3))\cong\Z/2$, there exists a unique non-trivial linear $S^2$-bundle over $S^5$, which we denote by \smash{$S^5\ttimes S^2\to S^5$}. By applying the Gysin sequence, it can be seen that the total space \smash{$S^5\ttimes S^2$} is a spin cohomology $S^2\times S^5$.
		
		To calculate its invariants, we first note that \smash{$S^5\ttimes S^2$} is diffeomorphic to the Aloff--Wallach space $\mathrm{SU}(3)/S^1$, where $S^1$ is embedded into $\mathrm{SU}(3)$ via
		\[ z\mapsto\begin{pmatrix}
			z & & \\ & z^{-1} & \\ & & 1
		\end{pmatrix}. \]
		Indeed, if we embedded $\mathrm{SU}(2)$ into $\mathrm{SU}(3)$ via $A\mapsto \left(\begin{smallmatrix}
			A &\\ & 1
		\end{smallmatrix}\right)$ and identify $\mathrm{SU}(2)\cong S^3$, the homogeneous fibration
		\[ S^3\cong \mathrm{SU}(2)\hookrightarrow\mathrm{SU}(3)\to\bigslant{\mathrm{SU}(3)}{\mathrm{SU}(2)}\cong S^5 \]
		is a principal $S^3$-bundle over $S^5$. Hence, we obtain a linear $S^2$-bundle over $S^5$ by taking the quotient of the fibre $\mathrm{SU}(2)$ by the circle subgroup
		\[ z\mapsto  \begin{pmatrix}
			z &\\&z^{-1}
		\end{pmatrix}. \]
		
		This shows that $\mathrm{SU}(3)/S^1$ is either diffeomorphic to $S^2\times S^5$ or to $S^5\ttimes S^2$. By \cite[Lemma~4.4]{KS91},\footnote{Note that there is a typo in the formula for $s_3$ in \cite[Lemma~4.4]{KS91}. The correct expression is given by ${s_3(N_{k,l})=\frac{-4P+NS}{6N}}$.} its $s$-invariants are given as follows:
		\[
         s_1\bigl(\mathrm{SU}(3)/S^1\bigr)=s_2\bigl(\mathrm{SU}(3)/S^1\bigr)=0,\qquad \tilde{s}_3\bigl(\mathrm{SU}(3)/S^1\bigr)=\frac{1}{2}.
         \]
		Hence, $\mathrm{SU}(3)/S^1$ is not diffeomorphic to $S^2\times S^5$ and consequently it is diffeomorphic to~\smash{${S^5\ttimes S^2}$}. In particular,
		\[
         s_1\bigl(S^5\ttimes S^2\bigr)=s_2\bigl(S^5\ttimes S^2\bigr)=0,\qquad \tilde{s}_3\bigl(S^5\ttimes S^2\bigr)=\frac{1}{2}.
         \]
		\item[$(3)$] Let $\Theta_7$ be the group of orientation-preserving diffeomorphism types of homotopy $7$-spheres. In \cite{EK62} an isomorphism
		\[
        \mu\colon\ \Theta_7\to\bigslant{\left(\frac{1}{28}\Z\right)}{\Z}
        \]
		of groups is constructed. If $\Sigma_r$ denotes the homotopy $7$-sphere with $\mu(\Sigma_r)=\frac{r}{28}$, then we~have
		\[
         s_1(M\#\Sigma_r)=s_1(M)+\frac{r}{28}
         \]
		for any spin cohomology $S^2\times S^5$ $M$ since the invariants $\mu$ and $s_1$ coincide for both homotopy $7$-spheres and spin cohomology $S^2\times S^5$s, see \cite[Section 3]{KS88} and \cite[Lemma~4]{Xu25}. Further, since $M$ and $M\#\Sigma_r$ are homeomorphic, we have
		\[
         s_2(M\#\Sigma_r)=s_2(M),\qquad \tilde{s}_3(M\#\Sigma_r)=\tilde{s}_3(M).
         \]
	\end{enumerate}
\end{Example}

\subsection{The diffeomorphism and homeomorphism classifications}

Items (1), (2a) and (2b) of Theorem~\ref{T:class} follow from the following.

\begin{Proposition}\label{P:s_i_values}
	For a spin cohomology $S^2\times S^5$, the invariants $s_1$, $s_2$, $\tilde{s}_3$ take values in $\bigl(\frac{1}{28}\Z\bigr)/\Z$, $\bigl(\frac{1}{12}\Z\bigr)/\Z$ and $\bigl(\frac{1}{2}\Z\bigr)/\Z$, respectively, and all values can be realised. %In particular, there are precisely $28\cdot 12\cdot 2=672$ oriented diffeomorphism types and $12\cdot 2=24$ oriented homeomorphism types of spin cohomology $S^2\times S^5$s.
\end{Proposition}
\begin{proof}
	It is shown in \cite[Section 5]{Wa22} that there are at most $672$ oriented diffeomorphism types (resp.\ $24$ oriented homeomorphism types) of spin cohomology $S^2\times S^5$s. Further, explicit examples are constructed where $s_2$ can take all possible values in $\bigl(\frac{1}{12}\Z\bigr)/\Z$, and $\tilde{s}_3$ can take all possible values in $\bigl(\frac{1}{2}\Z\bigr)/\Z$, while $s_1$ takes values in $\bigl(\frac{1}{28}\Z\bigr)/\Z$.
	
	Finally, we can realise all values of $\bigl(\frac{1}{28}\Z\bigr)/\Z$ for $s_1$ since, if $M$ is an arbitrary spin cohomology $S^2\times S^5$, we have
	\[ s_1(M\#\Sigma_r)=s_1(M)+\frac{r}{28}, \qquad s_2(M\#\Sigma_r)=s_2(M)\qquad  \text{and}\qquad \tilde{s}_3(M\#\Sigma_r)=\tilde{s}_3(M) \]
	by (3) of Example \ref{EX:spin_cohom_S2xS5}. Thus, these examples cover all possible values of $\bigl(\frac{1}{28}\Z\bigr)/\Z$, $\bigl(\frac{1}{12}\Z\bigr)/\Z$ and $\bigl(\frac{1}{2}\Z\bigr)/\Z$ for $s_1$, $s_2$ and $\tilde{s}_3$, respectively. In particular, since there exist at most $28\cdot 12\cdot 2=672$ oriented diffeomorphism types of spin cohomology $S^2\times S^5$s, these examples cover all oriented diffeomorphism types.
\end{proof}

\begin{proof}[Proof of items (1), (2a) and (2b) of Theorem~\ref{T:class}]
	Part (1) is precisely the content of Proposition \ref{P:s_i_values}. The diffeomorphism classification then follows from the fact that, by \cite[Theorem 1]{KS98}, the invariants $s_1$, $s_2$, $\tilde{s}_3$ uniquely determine the oriented diffeomorphism type. Similarly, for the homeomorphism classification, we use that, by \cite[Theorem 1]{KS98}, the invariants~$28s_1$,~$s_2$,~$\tilde{s}_3$ uniquely determine the oriented homeomorphism type. Since $s_1$ takes values in $\bigl(\frac{1}{28}\Z\bigr)/\Z$, we have $28s_1=0$ and hence it suffices to consider the invariants $s_2$ and $\tilde{s}_3$.
\end{proof}

\subsection{The homotopy classification}

To prove item (2c) of Theorem~\ref{T:class}, we will use the following:
\begin{Theorem}[{\cite[Theorem 0.1]{Kr98}}]\label{T:hom_class}
	Let $M$ and $M'$ be two spin cohomology $S^2\times S^5$s.
	\begin{enumerate}\itemsep=0pt
		\item[$(1)$] If $\pi_4(M)=\pi_4(M')=0$, then $M$ and $M'$ are orientation-preserving homotopy equivalent if and only if $2s_2(M)=2s_2(M')$.
		\item[$(2)$] If $\pi_4(M)\cong\pi_4(M')\cong\Z/2$, then $M$ and $M'$ are orientation-preserving homotopy equivalent if and only if $s_2(M)=s_2(M')$.
	\end{enumerate}
\end{Theorem}
Hence, it remains to determine the fourth homotopy group of each homeomorphism type. We will show that it is determined by $\tilde{s}_3$.
\begin{Proposition}\label{P:pi4}
	Let $M$ be a spin cohomology $S^2\times S^5$.
	\begin{enumerate}\itemsep=0pt
		\item[$(1)$] If $\tilde{s}_3(M)=0$, then $\pi_4(M)\cong\Z/2$.
		\item[$(2)$] If $\tilde{s}_3(M)=\frac{1}{2}$, then $\pi_4(M)=0$.
	\end{enumerate}
\end{Proposition}
Proposition \ref{P:pi4} together with Theorem~\ref{T:hom_class} directly implies item (3) of Theorem~\ref{T:class}. It therefore remains to prove Proposition \ref{P:pi4}. Since its proof requires the results of Section \ref{S:circ_bund}, we postpone it to Section~\ref{SS:pi4} below.

\section{Principal circle bundles}\label{S:circ_bund}

In this section, we prove the following two theorems, which directly imply Theorems \ref{T:free_circle} and \ref{T:Ric>0}.

\begin{Theorem}\label{T:free_circle_s}
	Let $M$ be a spin cohomology $S^2\times S^5$.
	\begin{enumerate}\itemsep=0pt
		\item[$(1)$] $M$ admits a free circle action with spin quotient space if and only if there exist $i,j\in\Z$ such that
$(s_1(M),s_2(M),\tilde{s}_3(M))=\bigl(\tfrac{i}{28}, \tfrac{j}{3},0 \bigr)$ or $(s_1(M),s_2(M),\tilde{s}_3(M))=\bigl(\tfrac{i}{28}, \tfrac{j}{12},\tfrac{1}{2} \bigr)$.
		\item[$(2)$] $M$ admits a free circle action with non-spin quotient space if and only if there exist $i,j\in\Z$ such that
$(s_1(M),s_2(M),\tilde{s}_3(M))=\bigl(\tfrac{i}{14}, \tfrac{j}{6},0 \bigr)$.
		\item[$(3)$] Whenever $M$ admits a free circle action, it admits infinitely many pairwise non-equivalent free circle actions.
	\end{enumerate}
\end{Theorem}

\begin{Theorem}\label{T:Ric>0_s}
	Let $M$ be a spin cohomology $S^2\times S^5$ whose $s$-invariants are of the form
	\begin{align*}
		(s_1(M),s_2(M),\tilde{s}_3(M))&= \bigl(\tfrac{i}{28}, \tfrac{j}{3},0 \bigr)\qquad\text{or}\\
		(s_1(M),s_2(M),\tilde{s}_3(M))&= \bigl(\tfrac{1+2i}{14}, \tfrac{1+2j}{6},0 \bigr)\qquad\text{or}\\
		%(s_1(M),s_2(M),\tilde{s}_3(M))&= \bigl(\tfrac{i}{28}, \tfrac{1+2j}{12},0 \bigr)\qquad\text{or}\\
		(s_1(M),s_2(M),\tilde{s}_3(M))&= \bigl(\tfrac{i}{28}, \tfrac{j}{12},\tfrac{1}{2} \bigr).%(s_1(M),s_2(M),\tilde{s}_3(M))&= \bigl(\tfrac{i}{28}, \tfrac{j}{6},\tfrac{1}{2} \bigr).
	\end{align*}
	Then $M$ admits infinitely many pairwise non-equivalent free circle actions and for each of these actions an invariant Riemannian metric of positive Ricci curvature.
\end{Theorem}

\begin{Remark}\label{R:Ric>0_missing}
	By comparing Theorems \ref{T:free_circle_s} and \ref{T:Ric>0_s}, we see that the spin cohomology $S^2\times S^5$s that admit a free circle action but are not covered by Theorem~\ref{T:Ric>0_s} are those with $s$-invariants
	\[ (s_1(M),s_2(M),\tilde{s}_3(M))=\bigl(\tfrac{i}{7},\tfrac{1+2j}{6},0\bigr), \]
	which consist of $7\cdot3=21$ oriented diffeomorphism types. It remains open whether these manifolds admit Riemannian metrics of positive Ricci curvature that are invariant under a free circle action, cf.\ Question~\ref{Q:Ric>0_missing}.
\end{Remark}

To prove Theorems \ref{T:free_circle_s} and \ref{T:Ric>0_s}, we first consider in Section~\ref{SS:quot_space} the possible $6$-manifolds that can arise as an orbit space of a free circle action on a spin cohomology $S^2\times S^5$. We then restrict to two types of manifolds: manifolds obtained as boundaries of certain plumbings of linear disc bundles over spheres (which are spin) and linear $S^2$-bundles over $\C P^2$. These are especially useful for the proof of Theorem~\ref{T:Ric>0_s} since both classes of manifolds admit Riemannian metrics of positive Ricci curvature. In Section~\ref{SS:s-inv}, we then compute the $s$-invariants of principal circle bundles over these manifolds and finish the proofs of Theorems~\ref{T:free_circle_s} and \ref{T:Ric>0_s}.

\subsection{Topology of the quotient space} \label{SS:quot_space}

The following result was established in \cite{Xu25}.

\begin{Lemma}[{\cite[Lemmas 1 and 2]{Xu25}}]\label{L:N_orbit_space}
	A closed $6$-manifold $N$ is the orbit space of a free circle action on a spin cohomology $S^2\times S^5$ with corresponding Euler class $e\in H^2(N)$ if and only if the following holds:
	\begin{enumerate}\itemsep=0pt
		\item[$(1)$] $N$ is simply-connected, has torsion-free homology and the only non-zero Betti numbers are $b_0(N)=b_6(N)=1$ and $b_2(N)=b_4(N)=2$.
		\item[$(2)$] The class $e$ can be extended to a basis $(e,f)$ of $H^2(N)$ such that for $A=\mu_N(e,e,e)$, $B=\mu_N(e,e,f)$, $C=\mu_N(e,f,f)$ we have $AC-B^2=\pm 1$.
		\item[$(3)$] $w_2(N)=e\mod 2$ when $N$ is non-spin.
	\end{enumerate}
	In this case, we additionally know the following. Here, $(e^*,f^*)$ denotes the dual basis of $(e,f)$ and we set $D=\mu_N(f,f,f)$.
	\begin{enumerate}\itemsep=0pt
		\item[$(a)$] When $N$ is spin, $p_1(N)=(24u+4A)e^*+(24v+4D)f^*$ for some $u,v\in\Z$, and $A$ is even and $B$, $C$ are odd.
		\item[$(b)$] When $N$ is non-spin, $p_1(N)=(48u+A)e^*+(24v+3B+6C+4D)f^*$ for some $u,v\in\Z$, and $A$, $C$ are even and $B$ is odd.
	\end{enumerate}
	Conversely, any set of integers $(A,B,C,D,u,v)$ satisfying $AC-B^2=\pm1$ and the parity conditions in $(a)$ $($resp.\ $(b))$ defines a unique closed, simply-connected spin $($resp.\ non-spin$)$ $6$-mani\-fold~$N$ with the properties {\rm (1)--(3)} and $(a)$ $($resp.\ $(b))$.
\end{Lemma}

Note that the value $AC-B^2$ is the determinant of the bilinear form on $H^2(N)$ defined by $(x,y)\mapsto \mu_N(x,y,e)$. It therefore does not depend on the choice of $f$, and we will denote it by $\det(N,e)$.

We will focus on two specific families of $6$-manifolds. The first one consists of manifolds constructed in \cite{Re24b} that are obtained by plumbing, and the second one consists of linear $S^2$-bundles over $\C P^2$.

\begin{Lemma}\label{L:N_A}
	For a triple $A=(\alpha_1,\alpha_2,\alpha_3)\in\Z^3$, there exists a closed, simply-connected spin $6$-manifold $N_A$ with the following properties:
	\begin{enumerate}\itemsep=0pt
		\item[$(1)$] The manifold $N_A$ has torsion-free homology and the only non-trivial Betti numbers are $b_0(N_A)=b_6(N_A)\allowbreak=1$ and $b_2(N_A)=b_4(N_A)=2$.
		\item[$(2)$] There exists a basis $(e_1,e_2)$ of $H^2(N_A)$ such that
		\begin{alignat*}{3}
			&\mu_{N_A}(e_1,e_1,e_1)=\alpha_1-\alpha_3,\qquad&&
			\mu_{N_A}(e_1,e_1,e_2)=-\alpha_3,&\\
			&\mu_{N_A}(e_1,e_2,e_2)=-\alpha_3,\qquad&&
			\mu_{N_A}(e_2,e_2,e_2)=\alpha_2-\alpha_3,&\\
			&\makebox[0pt][l]{$p_1(N_A)=4(\alpha_1-\alpha_3)e_1^*+4(\alpha_2-\alpha_3)e_2^*.$}&
		\end{alignat*}
	\end{enumerate}
\end{Lemma}
\begin{proof}
	This can be shown by verifying that the corresponding system of invariants satisfies~\eqref{EQ:Wall}. Alternatively, we can give an explicit geometric construction as follows.
	
	We consider the following bipartite graph:
	
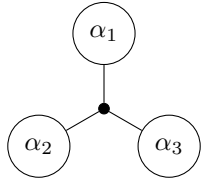
\begin{center}
		\begin{tikzpicture}
			\begin{scope}[every node/.style={circle,draw,minimum height=2.35em}]
				\node(U11) at (4,1){$\alpha_1$};
				\node(U21) at (3.133975,-0.5){$\alpha_2$};
				\node(U31) at (4.866025,-0.5){$\alpha_3$};
				\node[Bullet](V) at (4,0){};
			\end{scope}
			\path[-](V) edge (U11);
			\path[-](V) edge (U21);
			\path[-](V) edge (U31);
		\end{tikzpicture}
	\end{center}
	We now replace each labelled vertex with label $\alpha_i$ with the linear $D^3$-bundle over $S^4$ with first Pontryagin class $4\alpha_i[S^4]^*$, and we replace the unlabelled vertex with the trivial $D^4$-bundle $S^3\times D^4$ over $S^3$. Plumbing these bundles according to this graph results in a $7$-manifold whose boundary $N_A$ satisfies the conditions as claimed, see \cite[Lemma~4.6 and Proposition 4.12]{Re24b}.
\end{proof}

We note that the manifold $N_A$ remains unchanged under permutation (which corresponds to a~graph automorphism) and simultaneous multiplication by $(-1)$ (which corresponds to a change of orientation) of the triple $(\alpha_1,\alpha_2,\alpha_3)$. We will therefore call two triples $A,A'\in\Z^3$ \emph{equivalent} if one can be obtained from the other by a permutation possibly followed by a simultaneous multiplication by $(-1)$. It was shown in \cite[Proposition 4.12]{Re24b} that for any equivalence class of triples $[A]$ there can be at most one other equivalence class of triples $[A']$ such that $N_A$ and $N_{A'}$ are diffeomorphic. We will use this fact for the proof of part (2) of Theorem~\ref{T:free_circle}.

\begin{Lemma}\label{L:N_A_inv}
	Let $A=(\alpha_1,\alpha_2,\alpha_3)\in\Z^3$ and let $e=\lambda e_1+\mu e_2\in H^2(N_A)$ be primitive. Then $e$ can be extended to a basis $(e,f)$ of $H^2(N_A)$ such that {\rm (1)--(3)} of Lemma~$\ref{L:N_orbit_space}$ holds if and only if
	\[ \det(N_A,e)=-\lambda^2\alpha_1\alpha_3-\lambda\mu(\alpha_1\alpha_3+\alpha_2\alpha_3-\alpha_1\alpha_2)-\mu^2\alpha_2\alpha_3=\pm1. \]
	In this case, the values of $A$, $B$, $C$, $D$, $u$, $v$ are given as follows:
	\begin{align*}
		&A=\lambda^3(\alpha_1-\alpha_3)-3\lambda\mu(\lambda+\mu)\alpha_3+\mu^3(\alpha_2-\alpha_3),\\
		&B=\lambda^2(-t(\alpha_1-\alpha_3)-s\alpha_3)+2\lambda\mu(t-s)\alpha_3+\mu^2(t\alpha_3+s(\alpha_2-\alpha_3)),\\
		&C=\lambda\bigl(t^2(\alpha_1-\alpha_3)+2ts\alpha_3-s^2\alpha_3\bigr)+\mu\bigl(-t^2\alpha_3+2ts\alpha_3+s^2(\alpha_2-\alpha_3)\bigr),\\
		&D=-t^3(\alpha_1-\alpha_3)-3ts(t-s)\alpha_3+s^3(\alpha_2-\alpha_3),\\
		&u=\frac{\lambda(1-\lambda)(1+\lambda)}{6}(\alpha_1-\alpha_3)+\frac{\lambda\mu(\lambda+\mu)}{2}\alpha_3+\frac{\mu(1-\mu)(1+\mu)}{6}(\alpha_2-\alpha_3),\\
		&v=-\frac{t(1-t)(1+t)}{6}(\alpha_1-\alpha_3)+\frac{ts(t-s)}{2}\alpha_3+\frac{s(1-s)(1+s)}{6}(\alpha_2-\alpha_3).
	\end{align*}
	Here $s,t\in\Z$ are integers such that $s\lambda+t\mu=1$.
\end{Lemma}
\begin{proof}
	The proof is a tedious but straightforward calculation using Lemmas \ref{L:N_orbit_space} and \ref{L:N_A}. The calculation for the value of $\det(N_A,e)=AC-B^2$ can be simplified as follows: Since $AC-B^2$ is the determinant of the bilinear form $(x,y)\mapsto \mu_{N_A}(x,y,e)$ on $H^2(N_A)$, we can also compute it from the basis $(e_1,e_2)$ instead of $(e,f)$, i.e.,
\[
 AC-B^2=\mu_{N_A}(e_1,e_1,e)\mu_{N_A}(e_2,e_2,e)-\mu_{N_A}(e_1,e_2,e)^2. \tag*{\qed}
\]
\renewcommand{\qed}{}
\end{proof}

The second family of manifolds we consider is the following.

\begin{Lemma}%\label{L:N_ab}
	For a pair $(\alpha,\beta)\in\Z\times\{0,1\}$, there exists a closed, simply-connected $6$-manifold $N_{(\alpha,\beta)}$ with the following properties:
	\begin{enumerate}\itemsep=0pt
		\item[$(1)$] $N_{(\alpha,\beta)}$ has torsion-free homology and the only non-trivial Betti numbers are $b_0\bigl(N_{(\alpha,\beta)}\bigr)=b_6\bigl(N_{(\alpha,\beta)}\bigr)=1$ and $b_2\bigl(N_{(\alpha,\beta)}\bigr)=b_4\bigl(N_{(\alpha,\beta)}\bigr)=2$.
		\item[$(2)$] There exists a basis $(e_1,e_2)$ of $H^2\bigl(N_{(\alpha,\beta)}\bigr)$ such that
		\begin{alignat*}{3}
			&\mu_{N_{(\alpha,\beta)}}(e_1,e_1,e_1)=0,\qquad&&
			\mu_{N_{(\alpha,\beta)}}(e_1,e_1,e_2)=1,&\\
			&\mu_{N_{(\alpha,\beta)}}(e_1,e_2,e_2)=\beta,\qquad&&
			\mu_{N_{(\alpha,\beta)}}(e_2,e_2,e_2)=\alpha+\beta,&\\
			&\makebox[0pt][l]{$w_2\bigl(N_{(\alpha,\beta)}\bigr)=(1-\beta)e_1\mod 2,$}&\\
			&\makebox[0pt][l]{$p_1\bigl(N_{(\alpha,\beta)}\bigr)=(4\alpha+3+\beta)e_2^*.$}&
		\end{alignat*}
	\end{enumerate}
\end{Lemma}
\begin{proof}
	As in Lemma~\ref{L:N_A}, this can be shown by verifying that the corresponding system of invariants satisfies \eqref{EQ:Jupp}. We give the following alternative construction.
	
	We define $N_{\alpha,\beta}$ as the total space of the linear $S^2$-bundle over $\C P^2$ with first Pontryagin class \smash{$(4\alpha+\beta)\bigl[\C P^2\bigr]^*$} and second Stiefel--Whitney class $\beta x\mod 2$, where \smash{$x\in H^2\bigl(\C P^2\bigr)\cong \Z$} is a generator. The claim then follows from \cite[Corollary 5.8]{Re23}.
\end{proof}

By \cite[Proposition 4.12]{Re24b}, two different pairs $(\alpha,\beta)\neq (\alpha',\beta')$ define non-diffeomorphic manifolds \smash{$N_{(\alpha,\beta)}\not\cong N_{(\alpha',\beta')}$}. Moreover, there are diffeomorphisms
\begin{equation}\label{EQ:N_A_diff}
	N_{(\alpha,1)}\cong N_{(-1,-1,\alpha)}
\end{equation}
for all $\alpha\in\Z$, since the invariants coincide in this case (see also \cite[Lemma~4.17]{Re24b}). Therefore, for the purpose of defining orbit spaces for free circle actions on spin cohomology $S^2\times S^5$s, it will suffice to consider the non-spin manifolds $N_{(\alpha,0)}$. Similarly to Lemma~\ref{L:N_A_inv}, we obtain the~following.
\begin{Lemma}\label{L:N_ab_inv}
	Let $\alpha\in\Z$ and let $e=\lambda e_1+\mu e_2\in H^2\bigl(N_{(\alpha,0)}\bigr)$ be primitive. Then $e$ can be extended to a basis $(e,f)$ of $H^2\bigl(N_{(\alpha,0)}\bigr)$ such that {\rm (1)--(3)} of Lemma~$\ref{L:N_orbit_space}$ holds if and only if $\lambda$ is odd, $\mu$ is even, and
$\det\bigl(N_{(\alpha,0)}\bigr)=\mu^2\alpha-\lambda^2=-1$.
	In this case, the values of $A$, $B$, $C$, $D$, $u$, $v$ are given as follows:
	\begin{align*}
		&A=\mu(4\lambda^2-1),\qquad
		B=(4\lambda^2-1)s-2\lambda,\\
		&C=-3ts\lambda+t+\mu s^2\alpha,\qquad
		D=3t^2s+s^3\alpha,\\
		&u= \frac{\mu(1-\mu^2)\alpha}{12},\qquad
		v=\frac{\alpha s(1-s^2)}{6}-\frac{s\alpha\mu(2\mu-s)}{4}+\frac{\lambda-s+2t-2t^2s-3t^2\mu}{4}.
		%\frac{s(\alpha (1-s^2)-3t^2+3\lambda^2)}{6}+\frac{s+\lambda+3ts\lambda-t-\mu s^2\alpha}{4}
	\end{align*}
	Here $s,t\in\Z$ are integers such that $s\lambda+t\mu=1$.
\end{Lemma}
\begin{proof}
	The proof is again a straightforward calculation. Here we have $AC-B^2=\mu^2\alpha-\lambda^2$. Since (3) of Lemma~\ref{L:N_orbit_space} is equivalent to $\mu$ being even and $\lambda$ being odd, we have $\mu^2\alpha-\lambda^2\equiv -1\mod4$, hence we cannot have $\mu^2\alpha-\lambda^2=1$.
\end{proof}

\begin{Remark}
	A subfamily of spin cohomology $S^2\times S^5$s that admit a free circle action with quotient space of the form $N_{(\alpha,\beta)}$ can explicitly be described as follows. Let $a\in\Z$, $\varepsilon\in\{\pm1\}$ and let \smash{$E_a^\varepsilon\xrightarrow{\xi}\C P^2$} be the sphere bundle of the complex vector bundle of rank $2$ with Chern classes $c_1(\xi)=a x$ and $c_2(\xi)=\varepsilon x^2$, where $x\in H^2\bigl(\C P^2\bigr)\cong\Z$ denotes a generator and we refer to \cite[Theorem A]{Th74} for the existence of such bundles. In particular, $\xi$ is a linear $S^3$-bundle over $\C P^2$. Since the Euler class of $\xi$ is given by $e(\xi)=c_2(\xi)=\varepsilon x^2$, it follows from the Gysin sequence that $E_a^\varepsilon$ has the cohomology ring of $S^2\times S^5$. Further, when $a$ is odd, $E_a^\varepsilon$ is spin and hence $E_a^\varepsilon$ is a spin cohomology $S^2\times S^5$ as noted in \cite[Section 5]{KS25}. In addition, the $s$-invariants of $E_a^\varepsilon$ are given in \cite[(5.4)]{KS25}.
	
	Since $\xi$ is the sphere bundle of a complex vector bundle, the group $S^1\subseteq\C$ acts freely and fibre-preserving on $E_a^\varepsilon$. Since this action is the Hopf action on each fibre, we obtain that the quotient space is a linear $S^2$-bundle over $\C P^2$, i.e., a manifold of the form $N_{(\alpha,\beta)}$.
	
	To determine the values of $\alpha$ and $\beta$, we need to determine the second Stiefel--Whitney class and the first Pontryagin class of the linear $S^2$-bundle \smash{$N_{(\alpha,\beta)}\to\C P^2$}. For that, we use that, by construction, the classifying map $\C P^2\to\mathrm{BSO}(3)$ of this bundle factors through the composition
	\[ \C P^2\to\mathrm{BU}(2)\to\mathrm{BPU}(2)\cong\mathrm{BSO}(3), \]
	where $\C P^2\to\mathrm{BU}(2)$ is the classifying map of $\xi$, $\mathrm{BU}(2)\to\mathrm{BPU}(2)$ is induced by the quotient map $\mathrm{U}(2)\to\mathrm{U}(2)/\mathrm{U}(1)=\mathrm{PU}(2)$, and the homeomorphism $\mathrm{BPU}(2)\cong\mathrm{BSO}(3)$ is induced by the exceptional isomorphism $\mathrm{PU}(2)\cong\mathrm{SO}(3)$. Under this composition, the universal classes \smash{$w_2\in H^2(\mathrm{BSO}(3);\Z/2)$} and \smash{$p_1\in H^4(\mathrm{BSO}(3))$} map to
	\[c_1(\xi)\equiv x\mod 2\in H^2\bigl(\C P^2;\Z/2\bigr)\qquad
\text{and}\qquad c_1(\xi)^2-4c_2(\xi)= \bigl(a^2-4\varepsilon\bigr)x^2\in H^4\bigl(\C P^2\bigr), \]
	respectively, see \cite[Lemma~3.2]{AW14}. It follows that
	\[ \beta x\equiv  ax\equiv x\mod 2 \qquad
	\text{and}\qquad  (4\alpha+\beta)x^2=\bigl(a^2-4\varepsilon\bigr)x^2. \]
	Hence, \smash{$\alpha=\frac{a^2-1}{4}-\varepsilon$} and $\beta=1$.
\end{Remark}

\subsection[s-invariants of the total space]{$\boldsymbol{s}$-invariants of the total space}\label{SS:s-inv}

Given a closed $6$-manifold $N$ that satisfies (1)--(3) of Lemma~\ref{L:N_orbit_space}, the principal circle bundle over~$N$ with Euler class $e$ is a spin cohomology $S^2\times S^5$. Its $s$-invariants are given as follows.

\begin{Lemma}\label{L:s-inv}
	Let $N$ be a manifold as in Lemma~$\ref{L:N_orbit_space}$ and let $N(e)$ be the spin cohomology $S^2\times S^5$ that is the total space of the principal circle bundle over $N$ with Euler class $e$. Then the $s$-invariants of $N(e)$ are given as follows:
	\begin{enumerate}\itemsep=0pt
		\item[$(1)$] If $N$ is spin with $\det(N,e)=AC-B^2=-1$, then
		\begin{gather*}
			s_1(N(e))= -\frac{9}{14}\bigl(Cu^2-2Buv+Av^2\bigr)+\frac{2-3B(B-D)}{14}u+\frac{3A(B-D)}{14}v\\
			 \hphantom{s_1(N(e))=}{} +\frac{A}{224}-\frac{A}{56}(B-D)^2\mod 1,\\
			s_2(N(e))= \frac{D+1}{2}u+A\frac{2C^2-2BD+D^2}{24}\mod 1,\\
			\tilde{s}_3(N(e))= \frac{D+1}{2}\mod 1.
		\end{gather*}
		\item[$(2)$] If $N$ is spin with $\det(N,e)=AC-B^2=1$, then
		\begin{gather*}
			s_1(N(e))= -\frac{\mathrm{sgn}(A)}{112}+\frac{9}{14}\bigl(Cu^2-2Buv+Av^2\bigr)+\frac{2+3B(B-D)}{14}u-\frac{3A(B-D)}{14}v\\
			\hphantom{s_1(N(e))=}{} +\frac{A}{224}+\frac{A}{56}(B-D)^2\mod 1,\\
			s_2(N(e))= \frac{D+1}{2}u-A\frac{2C^2-2BD+D^2}{24}+C\frac{1+C^2}{24}\mod 1,\\
			\tilde{s}_3(N(e))= \frac{D+1}{2}\mod 1.
		\end{gather*}
		\item[$(3)$] If $N$ is non-spin, then $\det(N,e)=AC-B^2=-1$ and
		\begin{gather*}
			s_1(N(e))= -\frac{9}{14}\bigl(4Cu^2-4Buv+Av^2\bigr)+\frac{3}{14}B(B+3C+2D)u\\
			\hphantom{s_1(N(e))=}{} -\frac{3}{28}A(B+3C+2D)v\\
			\hphantom{s_1(N(e))=}{} -\frac{1}{224}A\bigl(AC+6BC+4BD+9C^2+4D^2+12CD\bigr)\mod 1,\\
			s_2(N(e))= -\frac{1}{24}\bigl(B^2C+3BC^2-ABD-3ACD-AD^2+C+C^3\bigr)\mod 1,\\
			\tilde{s}_3(N(e))= 0\mod 1.
		\end{gather*}
	\end{enumerate}
\end{Lemma}
\begin{proof}
	Items (1) and (3) are the content of \cite[Lemma~3]{Xu25}. For item (2), it is shown in the proof of \cite[Lemma~3]{Xu25} that
	\begin{gather*}
		s_1(N(e))= -\frac{1}{224}\sigma(M_e)+\frac{1}{896}\biggl(A+2k+\begin{pmatrix}
			k & l
		\end{pmatrix}M_e^{-1}\begin{pmatrix}
			k \\ l
		\end{pmatrix}\biggr)\\
		\hphantom{s_1(N(e))=}{} -\frac{1}{192}(k+A)+\frac{1}{384}A\mod 1,\\
		s_2(N(e))= \frac{1}{48}\biggl(-\begin{pmatrix}
			C & D
		\end{pmatrix}M_e^{-1}\begin{pmatrix}
			k\\l
		\end{pmatrix}-l+2\begin{pmatrix}
			C & D
		\end{pmatrix}M_e^{-1}\begin{pmatrix}
			C\\D
		\end{pmatrix}+4D+2C \biggr)\mod 1,\\
		\tilde{s}_3(N(e))= \frac{1}{24}l+\frac{1}{2}\begin{pmatrix}
			C & D
		\end{pmatrix}M_e^{-1}\begin{pmatrix}
			C\\D
		\end{pmatrix}+\frac{1}{3}D\mod 1.
	\end{gather*}
	Here, $k=4A+24u$, $l=4D+24v$ and $M_e$ is the matrix
	\[ M_e=\begin{pmatrix}
		A & B\\B & C
	\end{pmatrix}. \]
	When $\det(N,e)=\det(M_e)=1$, the signature $\sigma(M_e)$ equals $\sigma(M_e)=\mathrm{sgn}(A)$ since $A$ and $C$ have the same sign and so
	\[ \sigma(M_e)=2\mathrm{sgn}(\mathrm{tr}(M_e))=2\mathrm{sgn}(A). \]
	Further,
	\[ M_e^{-1}=\begin{pmatrix}
		\hphantom{-}C & -B\\-B& \hphantom{-}A\end{pmatrix}. \]
	A tedious but straightforward calculation using that $A$ is even, $B,C$ are odd and $AC-B^2=1$ then finishes the proof.
\end{proof}	

Lemma~\ref{L:s-inv} allows us to prove the following result.
\begin{Lemma}\label{L:s_inv_nec}
	In the setting of Lemma~$\ref{L:s-inv}$, the possible values of $s_2(N(e))$ and $\tilde{s}_3(N(e))$ are as follows:
	\begin{enumerate}\itemsep=0pt
		\item If $N$ is spin with $\det(N,e)=-1$, then
		\begin{gather*}
         (s_2(N(e)),\tilde{s}_3(N(e)))=\bigl(\tfrac{j}{3},0\bigr),\qquad j=0,1,2,\qquad \text{or}\\
          (s_2(N(e)),\tilde{s}_3(N(e)))=\bigl(\tfrac{j}{6},\tfrac{1}{2}\bigr),\qquad j=0,\dots,5.
        \end{gather*}
		\item If $N$ is spin with $\det(N,e)=1$, then
		\begin{gather*}
         (s_2(N(e)),\tilde{s}_3(N(e)))=\bigl(\tfrac{j}{3},0\bigr), \qquad j=0,1,2,\qquad \text{or}\\
          (s_2(N(e)),\tilde{s}_3(N(e)))=\bigl(\tfrac{1+2j}{12},\tfrac{1}{2}\bigr),\qquad j=0,\dots,5.
         \end{gather*}
		\item If $N$ is non-spin, then
		\[(s_1(N(e)),s_2(N(e)),\tilde{s}_3(N(e)))=\bigl(\tfrac{k}{14},\tfrac{j}{6},0\bigr),\qquad k=0,\dots,13,\  j=0,\dots,5.\]
	\end{enumerate}
	In particular, there are at most $28\cdot (3+6+6)=420$ possible values of the $s$-invariants $(s_1(N(e)),s_2(N(e)),\tilde{s}_3(N(e)))$ when $N$ is spin, $6\cdot 14=84$ when $N$ is non-spin, and $420+3\cdot 14=462$ in total.
\end{Lemma}
\begin{proof}
	First, suppose that $N$ is spin with $AC-B^2=-1$. Since $B$ is odd, we can write it as $B=2B'+1$ and hence
	\[ AC=B^2-1=4B'\bigl(B'+1\bigr). \]
	Since $C$ is odd, it follows that $A$ is divisible by $8$. Hence, if $\tilde{s}_3(N(e))=0$, we obtain
	\[ s_2(N(e))=\frac{A}{8}\frac{2C^2-2BD+D^2}{3}=\frac{j}{3} \]
	for $j=\frac{A}{8}\bigl(2C^2-2BD+D^2\bigr)$, and if $\tilde{s}_3(N(e))=\frac{1}{2}$, we obtain
	\[ s_2(N(e))=\frac{3(D+1)u}{6}+2\frac{A}{8}\frac{2C^2-2BD+D^2}{6}=\frac{j}{6} \]
	for $j=3(D+1)u+2\frac{A}{8}\bigl(2C^2-2BD+D^2\bigr)$.
	
	Next, suppose that $N$ is spin with $AC-B^2=1$. If $\tilde{s}_3(N(e))=0$, we have that $D$ is odd and we obtain
	\begin{align*}
    s_2(N(e))&=-A\frac{2C^2-2BD+D^2}{24}+C\frac{1+C^2}{24}\\
    &=\frac{1}{12}\biggl( -\frac{A}{2}\bigl(2C^2-2BD+D^2\bigr)+C\frac{1+C^2}{2} \biggr).
    \end{align*}
	We now show that the term in the bracket is divisible by $4$ and hence \smash{$s_2(N(e))=\frac{j}{3}$}. Indeed, since $B$, $C$, $D$ are odd, we have $C^2\equiv 1\mod 4$, $BD\equiv\pm 1\mod 4$ and $D^2\equiv 1\mod 4$. Hence,
	\[ 2C^2-2BD+D^2\equiv 2\pm 2+1\equiv 1\mod 4. \]
	Moreover, since $C$ is odd, we have $C^2\equiv 1\mod 8$ and therefore \smash{$\frac{1+C^2}{2}\equiv 1\mod 4$}. Finally, since $AC=B^2+1\equiv 2\mod 8$, so \smash{$\frac{A}{2}C\equiv 1\mod 4$}, we have that $\frac{A}{2}\equiv C\mod 4$. It follows that
	\[-\frac{A}{2}\bigl(2C^2-2BD+D^2\bigr)+C\frac{1+C^2}{2}\equiv -\frac{A}{2}+C\equiv 0\mod 4.  \]
	
	If \smash{$\tilde{s}_3(N(e))=\frac{1}{2}$}, we have that $D$ is even and again by using that \smash{$\frac{1+C^2}{2}\equiv 1\mod 4$} we obtain
	\[ s_2(N(e))=\frac{D+1}{2}u-A\frac{2C^2-2BD+D^2}{24}+C\frac{1+C^2}{24}=\frac{1+2j}{12} \]
	for
	\[j=3(D+1)u-\frac{A}{2}\biggl(C^2-BD+\frac{D^2}{2}\biggr)+\frac{\frac{C(1+C^2)}{2}-1}{2}.\]
	
	Finally, suppose that $N$ is non-spin. Then $B$ is odd while $A$ and $C$ are even. In particular, $B^2\equiv 1\mod 8$ and all of $A^2$, $AC$, $C^2$ are divisible by $4$. It follows that
	\begin{align*}
\begin{split}
		&B^2C+3BC^2-ABD-3ACD-AD^2+C+C^3\\
& \qquad \equiv C-ABD-AD^2+C
		 =2C-AD(B+D)
		 \equiv -AD(B+D)\mod 4.
\end{split}
	\end{align*}
	Since $B$ is odd, $D(B+D)$ is even, so $-AD(B+D)\equiv 0\mod 4$. Hence, $s_2(N(e))$ is of the form~$\frac{j}{6}$.
	
	To analyse $s_1(N(e))$ when $N$ is non-spin, we write $s_1(N(e))=\frac{k}{28}$ and we obtain
	\begin{align*}
		k={}&{-}18\bigl(4Cu^2-4Buv+Av^2\bigr)+6B(B+3C+2D)u-3A(B+3C+2D)v\\
		&{-}\,\frac{1}{2}AD(B+D)-\frac{1}{8}AC(A+6B+9C+12D).
	\end{align*}
	Since we have $AC-B^2=-1$, it follows as in the first case that $AC$ is divisible by $8$. Moreover, since $B$ is odd, $D(B+D)$ is even as above. Hence, every term in this expression is even and therefore $k$ is even.		
\end{proof}

\subsection{Proof of the main results}

We can now give the proof of Theorems \ref{T:free_circle_s} and \ref{T:Ric>0_s}. To show that all $s$-invariants in Lemma~\ref{L:s_inv_nec} can in fact be realised, we consider the manifolds $N_A$ and $N_{(\alpha,0)}$.
\begin{Proposition}\label{P:s_inv_real}
	Let $M$ be a spin cohomology $S^2\times S^5$.
	\begin{enumerate}\itemsep=0pt
		\item[$(1)$] If $(s_1(M),s_2(M),\tilde{s}_3(M))=\bigl(\tfrac{i}{28},\tfrac{j}{3},0\bigr)$ or
        $(s_1(M),s_2(M),\tilde{s}_3(M))=\bigl(\tfrac{i}{28},\tfrac{j}{6},\tfrac{1}{2}\bigr)$,
		then there exist $A=(\alpha_1,\alpha_2,\alpha_3)$ and $e\in H^2(N_A)$ such that $\det(N_A,e)=-1$ and ${M\cong N_A(e)}$.
		\item[$(2)$] If
$(s_1(M),s_2(M),\tilde{s}_3(M))=\bigl(\tfrac{i}{28},\tfrac{j}{3},0\bigr)$ or $(s_1(M),s_2(M),\tilde{s}_3(M))=\bigl(\tfrac{i}{28},\tfrac{1+2j}{12},\tfrac{1}{2}\bigr)$,
		then there exist $A=(\alpha_1,\alpha_2,\alpha_3)$ and $e\in H^2(N_A)$ such that $\det(N_A,e)=1$ and ${M\cong N_A(e)}$.
		\item[$(3)$] If
$(s_1(M),s_2(M),\tilde{s}_3(M))=\bigl(\tfrac{i}{14},\tfrac{j}{3},0\bigr)$ or
        $(s_1(M),s_2(M),\tilde{s}_3(M))=\bigl(\tfrac{1+2i}{14},\tfrac{1+2j}{6},0\bigr)$,
then there exist $\alpha\in\Z$ and $e\in H^2\bigl(N_{(\alpha,0)}\bigr)$ such that $M\cong N_{(\alpha,0)}(e)$.
	\end{enumerate}
\end{Proposition}
\begin{proof}
	To calculate the $s$-invariants of the total spaces $N_A(e)$ and $N_{(\alpha,0)}(e)$, one can combine Lemma~\ref{L:s-inv} with Lemmas \ref{L:N_A_inv} and \ref{L:N_ab_inv}. However, the resulting formulae appear to be highly complicated (see, e.g., \cite[Propositions 5.2 and 7.11]{EZ14} for some special cases). On the other hand, a simple computer search shows that all the values as claimed are realised. The results of such a search are given in Tables \ref{TA:det=-1}--\ref{TA:N_alpha} below.
\end{proof}

Finally, for part (3) of Theorem~\ref{T:free_circle_s}, we need the following algebraic result.
\begin{Lemma}\label{L:number_theory}
	Let $T\in\N$ and $S\in\Z$. Then
	\begin{enumerate}\itemsep=0pt
		\item[$(1)$] For any integer solution $(\lambda,\mu,\alpha_1,\alpha_2,\alpha_3)$ of the equation
		\begin{equation}\label{EQ:la-mu}
			\lambda\mu\alpha_1\alpha_2-\lambda(\lambda+\mu)\alpha_1\alpha_3-\mu(\lambda+\mu)\alpha_2\alpha_3=S,
		\end{equation}
		there exist infinitely many integer solutions of this equation in the congruence class of $(\lambda,\mu,\alpha_1,\alpha_2,\alpha_3)\mod T$.
		\item[$(2)$] For any integer solution $(\lambda,\mu,\alpha)$ of the equation
		\begin{equation}\label{EQ:la-mu-a}
			\mu^2\alpha-\lambda^2=S,
		\end{equation}
		there exist infinitely many integer solutions of this equation in the congruence class of $(\lambda,\mu,\alpha)\mod T$.
	\end{enumerate}
\end{Lemma}
\begin{proof}
		(1) First assume that $\lambda,\mu,(\lambda+\mu)\neq 0$. We set $A=\lambda\alpha_1$, $B=\mu\alpha_2$ and $C=(\lambda+\mu)\alpha_3$. Then \eqref{EQ:la-mu} is equivalent to
		\[ AB-AC-BC=S, \]
		which in turn is equivalent to
		\[ (A-C)(B-C)-C^2=S. \]
		For $n\in\Z$, we define $A_n=A+n(A-C)$, $B_n=B+n(2C+(n+1)(A-C))$ and $C_n=C+n(A-C)$. Then
		\begin{align*}
			 &(A_n-C_n)(B_n-C_n)-C_n^2\\
&\qquad =(A-C)(B-C+n(2C+n(A-C))) - (C+n(A-C))^2\\
&\qquad			= (A-C)(B-C)-C^2+n(A-C)(2C+n(A-C)) - 2nC(A-C)+n^2(A-C)^2\\		
&\qquad				= S.
		\end{align*}
		Hence, for $n=m\lambda\mu(\lambda+\mu)T$ with $m\in\Z$, we obtain for
		\begin{align*}
			&\lambda^{(m)}=\lambda,\qquad \mu^{(m)}=\mu,\\
			&\alpha_1^{(m)}=\alpha_1+m\mu(\lambda+\mu)(A-C)T,\\
			&\alpha_2^{(m)}=\alpha_2+m\lambda(\lambda+\mu)(2C+(n+1)(A-C))T,\\
			&\alpha_3^{(m)}=\alpha_3+m\lambda\mu(A-C)T
		\end{align*}
		a sequence of solutions of \eqref{EQ:la-mu} in the same congruence class of $(\lambda,\mu,\alpha_1,\alpha_2,\alpha_3)\mod T$.
		
		Since we assumed that $\lambda,\mu,(\lambda+\mu)\neq 0$, this sequence is constant in $m$ only when $A-C=0$ and $C=0$, i.e., $A=C=0$ and therefore $\alpha_1=\alpha_3=0$. In this case the left-hand side of \eqref{EQ:la-mu} does not depend on $\alpha_2$, so we can freely choose its value. In particular we can assign infinitely many values to $\alpha_2$ while leaving it unchanged $\mod T$.
		
		Finally, assume that one of $\lambda$, $\mu$, $(\lambda+\mu)$, say $(\lambda+\mu)$, vanishes. In this case the left-hand side of \eqref{EQ:la-mu} does not depend on $\alpha_3$, and as before we can freely choose its value.
		
        (2) First assume $\mu\neq0$. Then, for $m\in\Z$ we set
		\begin{align*}
			&\lambda^{(m)}=\lambda+\mu^2m T,\qquad
            \mu^{(m)}=\mu,\qquad
			\alpha^{(m)}=\alpha+\bigl(2\lambda+\mu^2mT\bigr)mT.
		\end{align*}
		A calculation shows that $\bigl(\lambda^{(m)},\mu^{(m)},\alpha^{(m)}\bigr)$ is a sequence of solutions of \eqref{EQ:la-mu-a} in the same congruence class of $(\lambda,\mu,\alpha)\mod T$. Since $\mu\neq 0$, this sequence is non-constant.
		
		Finally, if $\mu=0$, the left-hand side of \eqref{EQ:la-mu-a} does not depend on $\alpha$, so we can freely choose its value while leaving it unchanged $\mod T$.
\end{proof}

\begin{proof}[Proof of Theorem~\ref{T:free_circle_s}]
	The fact that the given $s$-invariants are the only possible $s$-invariants of a spin cohomology $S^2\times S^5$ with a free circle action is the content of Lemma~ \ref{L:s_inv_nec}. These values are also realised by Proposition \ref{P:s_inv_real}, except if
	\[ (s_1(M),s_2(M),\tilde{s}_3(M))=\bigl( \tfrac{i}{7},\tfrac{1+2j}{6},0 \bigr). \]
	In this case, we now construct a closed, simply-connected non-spin $6$-manifold $N$ with torsion-free homology such that $M\cong N(e)$ for some $e\in H^2(N)$.
	
	Set $B=1$, $C=0$, $D=1$ and $v=0$, and let $A\in\Z$ with $A\equiv 2\mod4$ and $u\in\Z$ be odd. Then, by Lemma~\ref{L:N_orbit_space}, there is a unique closed, simply-connected non-spin $6$-manifold $N$ realising these values. Let $e\in H^2(N)$ as in Lemma~\ref{L:N_orbit_space}. Then, by Lemma~\ref{L:s-inv}, the manifold~$N(e)$ is a~spin cohomology $S^2\times S^5$ with invariants
	\begin{align*}
		&s_1(N(e))=\frac{1}{7}\biggl( \frac{9}{2}u-\frac{A}{4} \biggr),\qquad
		s_2(N(e))=\frac{A}{12},\qquad
		\tilde{s}_3(N(e))=0.
	\end{align*}
	Since $\frac{A}{2}$ is odd, $6s_2(N(e))$ is an odd integer, and any odd integer $\mod6$ can be realised by choosing $A$ appropriately. Since $u$ is odd, $7s_1(N(e))$ is an integer, and by choosing $u$ appropriately, any integer $\mod7$ can be realised.
	
	Moreover, we have $p_1(N)(e)=48u+A$. Hence, since for each fixed triple $\bigl(\tfrac{i}{7},\frac{1+2j}{6},0\bigr)$ of invariants we can choose $u$ arbitrarily large, there are an infinite number of values realising these values. This proves Theorem~\ref{T:free_circle_s} in this case.
	
	It remains to prove part (3) of Theorem~\ref{T:free_circle_s} when the orbit space is of the form~$N_A$ or~\smash{$N_{(\alpha,0)}$}. Suppose first that $M$ is a spin cohomology $S^2\times S^5$ that admits a free circle action with orbit space of the form $N_A$. By Lemmas \ref{L:N_A_inv} and \ref{L:s-inv}, the $s$-invariants of $N_A(e)$ are polynomials in $(\lambda,\mu,\alpha_1,\alpha_2,\alpha_3)$ with rational coefficients. Let $T$ be the least common multiple of the denominators of these coefficients. Then, by item (1) of Lemma~\ref{L:number_theory}, there exist infinitely many integer tuples $\bigl(\lambda',\mu',\alpha_1',\alpha_2',\alpha_3'\bigr)$ with
	\begin{gather*} \lambda'\mu'\alpha_1'\alpha_2'-\lambda'\bigl(\lambda'+\mu'\bigr)\alpha_1'\alpha_3'-\mu'\bigl(\lambda'+\mu'\bigr)\alpha_2'\alpha_3'\\
     \qquad=\lambda\mu\alpha_1\alpha_2-\lambda(\lambda+\mu)\alpha_1\alpha_3-\mu(\lambda+\mu)\alpha_2\alpha_3=\pm 1
    \end{gather*}
	in the same congruence class $\mod T$. It follows that there are infinitely many tuples $A'=\bigl(\lambda',\mu',\alpha_1',\alpha_2',\alpha_3'\bigr)$ defining manifolds $N_{A'}$ and classes $e'\in H^2(N_{A'})$ such that $N_{A'}(e')$ is a~spin cohomology $S^2\times S^5$ with the same $s$-invariants as $M$, hence $N_{A'}(e')\cong M$. By \cite[Proposition~4.12]{Re24b}, there are only finitely many triples $(\alpha_1,\alpha_2,\alpha_3)$ defining diffeomorphic manifolds~$N_A$. Hence, there are infinitely many pairwise non-diffeomorphic manifolds $N_{A'}$ and classes $e'\in H^2(N_{A'})$ with $N_{A'}(e')\cong M$. By Proposition \ref{P:action=bdl}, this proves part (3) in this case.
	
	Finally, let $M$ be a spin cohomology $S^2\times S^5$ that admits a free circle action with orbit space of the form \smash{$N_{(\alpha,0)}$}. Here the argument goes along the same lines as in the previous case, where we use item (2) of Lemma~\ref{L:number_theory} to construct infinitely many values of $(\lambda,\mu,\alpha)$, and we use \cite[Proposition~4.12]{Re24b} to deduce that $N_{(\alpha,0)}$ and $N_{(\alpha',0)}$ are not diffeomorphic when $\alpha\neq\pm\alpha'$.
\end{proof}

\begin{proof}[Proof of Theorem~\ref{T:Ric>0_s}]
	As seen in the proof of Theorem~\ref{T:free_circle_s}, when $M$ is a spin cohomology $S^2\times S^5$ with $s$-invariants as in Theorem~\ref{T:Ric>0_s}, then $M$ admits infinitely many pairwise non-equivalent free circle actions such that each of these actions has quotient space of the form $N_A$ or $N_{(\alpha,0)}$. Since manifolds of this form admit a Riemannian metric of positive Ricci curvature by \cite[Theorem B]{Re23} and \cite[Theorem 2.7.3]{GW09}, the claim follows from Theorem~\ref{T:Ric>0_bdl}.
\end{proof}

\subsection{Proof of Proposition \ref{P:pi4}}\label{SS:pi4}

Finally, we can give the proof of Proposition \ref{P:pi4}. For that, we first calculate the $s$-invariants of total spaces of principal circle bundles over the manifolds $N_{(-1,-1,\alpha)}$ and $N_{(\alpha,0)}$. Recall that, by~\eqref{EQ:N_A_diff}, these are precisely all total spaces of linear $S^2$-bundles over $\C P^2$.
\begin{Lemma}\label{L:s_inv_S2-bdls}
	Let $\alpha\in\Z$.
	\begin{enumerate}\itemsep=0pt
		\item[$(1)$] We consider the manifold $N_{A}$ with $A=(-1,-1,\alpha)$ and set $e=e_1-e_2$. Then $N_{A}(e)$ is a~spin cohomology $S^2\times S^5$ with $s$-invariants
		\[ s_1(N_A(e))=s_2(N_A(e))=0,\qquad \tilde{s}_3(N_A(e))=\frac{\alpha}{2}\mod 1. \]
		In particular, $N_A(e)\cong S^2\times S^5$ when $\alpha$ is even and \smash{$N_A(e)\cong S^5\ttimes S^2$} when $\alpha$ is odd.
		\item[$(2)$] We consider the manifold $N_{(\alpha,0)}$ and set $e=e_1$. Then $N_{(\alpha,0)}(e)$ is a spin cohomology $S^2\times S^5$ with $s$-invariants
		\[ s_1\bigl(N_{(\alpha,0)}(e)\bigr)=s_2\bigl(N_{(\alpha,0)}(e)\bigr)=\tilde{s}_3\bigl(N_{(\alpha,0)}\bigr)=0. \]
		In particular, $N_{(\alpha,0)}(e)\cong S^2\times S^5$.
	\end{enumerate}
\end{Lemma}
\begin{proof}
	The proof is a tedious but straightforward calculation using Lemmas \ref{L:N_A_inv}, \ref{L:N_ab_inv} and \ref{L:s-inv}. The values of $(A,B,C,D,u,v)$ are given by $(0,-1,-1,-\alpha-1,0,0)$ (with $(s,t)=(0,-1)$) in the first case and by $(0,1,0,\alpha,0,0)$ (with $(s,t)=(1,0)$) in the second case. To identify the total spaces with $S^2\times S^5$ and $S^5\ttimes S^2\cong\mathrm{SU}(3)$, we use Example~\ref{EX:spin_cohom_S2xS5}.
\end{proof}

\begin{Remark}
	In combination with Lemma~\ref{L:princ_circ_hom}, Lemma~\ref{L:s_inv_S2-bdls} allows us to compute the homotopy groups of the manifolds $N_{(\alpha,\beta)}$, which might be of independent interest. Indeed, since the manifolds $N_{(\alpha,\beta)}$ are simply-connected and \smash{$H_2\bigl(N_{(\alpha,\beta)}\bigr)\cong\Z^2$}, we have
	\[ \pi_1\bigl(N_{(\alpha,\beta)}\bigr)=0,\qquad\pi_2\bigl(N_{(\alpha,\beta)}\bigr)\cong\Z^2. \]
	Further, since, by Example \ref{EX:spin_cohom_S2xS5}, \smash{$S^5\ttimes S^2$} is an $S^1$-quotient of $\mathrm{SU}(3)$, it follows from Lemmas \ref{L:princ_circ_hom} and \ref{L:s_inv_S2-bdls} (and using \eqref{EQ:N_A_diff}) that for $i\geq3$,
	\[ \pi_i\bigl(N_{(\alpha,\beta)}\bigr)\cong \pi_i\bigl(S^2\times S^5\bigr)\cong\pi_i\bigl(S^2\bigr)\oplus\pi_i\bigl(S^5\bigr)\qquad \text{for }\beta=0,\text{ or }\beta=1\text{ and }\alpha\text{ even}, \]
	and
	\[ \pi_i\bigl(N_{(\alpha,\beta)}\bigr)\cong \pi_i(\mathrm{SU}(3))\qquad \text{for }\beta=1\text{ and }\alpha\text{ odd}. \]
\end{Remark}

\begin{proof}[Proof of Proposition \ref{P:pi4}]
	We calculate $\pi_4$ for each of the $24$ oriented homeomorphism types of spin cohomology $S^2\times S^5$s.
	
	We first consider the case \smash{$\tilde{s}_3=\frac{1}{2}$}. Tables \ref{TA:det=-1} and \ref{TA:det=1} show that all $12$ oriented homeomorphism types can be realised by manifolds of the form $N_{(-1,-1,\alpha)}(e)$. More precisely, we extract the following from Tables~\ref{TA:det=-1} and~\ref{TA:det=1}, respectively.

	\vspace{0.3cm}

	\begin{paracol}{2}\centering
		\begin{supertabular}{|w{c}{.5cm}w{c}{.5cm}w{c}{.5cm}w{c}{.5cm}w{c}{.5cm}|w{c}{.5cm}w{c}{.5cm}w{c}{.5cm}|}
			\hline $\alpha_1$ & $\alpha_2$ & $\alpha_3$ & $\lambda$ & $\mu$ & $28s_1$ & $12s_2$ & $2\tilde{s}_3$ \\ \hline
			$-$1 & $-$1 & $-$1 & 0 & $-$1 & 0 & 0 & 1 \\
			$-$1 & $-$1 & 1 & 1 & $-$2 & 25 & 2 & 1\\
			$-$1 & $-$1 & 1 & 5 & $-$13 & 11 & 4 & 1\\
			$-$1 & $-$1 & 1 & 2 & $-$5 & 14 & 6 & 1\\
			$-$1 & $-$1 & 19 & 4 & $-$5 & 14 & 8 & 1\\
			$-$1 & $-$1 & 1 & 2 & $-$1 & 3 & 10 & 1\\\hline
		\end{supertabular}
		
		\switchcolumn
		
		\begin{supertabular}{|w{c}{.5cm}w{c}{.5cm}w{c}{.5cm}w{c}{.5cm}w{c}{.5cm}|w{c}{.5cm}w{c}{.5cm}w{c}{.5cm}|}
			\hline $\alpha_1$ & $\alpha_2$ & $\alpha_3$ & $\lambda$ & $\mu$ & $28s_1$ & $12s_2$ & $2\tilde{s}_3$ \\ \hline
			$-$1 & $-$1 & 13 & 4 & $-$3 & 21 & 1 & 1\\
			$-$1 & $-$1 & 3 & 2 & $-$1 & 22 & 3 & 1\\
			$-$1 & $-$1 & 7 & 2 & $-$3 & 0 & 5 & 1\\
			$-$1 & $-$1 & 7 & 3 & $-$2 & 0 & 7 & 1\\
			$-$1 & $-$1 & 21 & 5 & $-$4 & 21 & 9 & 1\\
			$-$1 & $-$1 & 1 & 0 & $-$1 & 1 & 11 & 1\\ \hline
		\end{supertabular}
	\end{paracol}

	This shows that all oriented homeomorphism types $M$ with $(s_2(M),\tilde{s}_3(M))=\bigl(\frac{j}{12},\frac{1}{2}\bigr)$ are of the form $N_{(-1,-1,\alpha)}(e)$ with $\alpha$ odd. By Lemmas \ref{L:princ_circ_hom} and \ref{L:s_inv_S2-bdls}, we have \[\pi_4(N_{(-1,-1,\alpha)})\cong \pi_4\bigl(S^5\ttimes S^2\bigr)\cong \pi_4(\mathrm{SU}(3))=0\] whenever $\alpha$ is odd. Hence, $\pi_4(M)=0$ by Lemma~\ref{L:princ_circ_hom}.
	
	Now assume that $M$ is a spin cohomology $S^2\times S^5$ with $\tilde{s}_3(M)=0$. It is shown in \cite[p. 471]{Kr98} that either $\pi_4(M)=0$ or $\pi_4(M)\cong\Z/2$. Assume $\pi_4(M)=0$. Let $M'$ be the spin cohomology $S^2\times S^5$ with invariants
	\[ \big(s_1\big(M'\big),s_2\big(M'\big),\tilde{s}_3\big(M'\big)\big)=\bigl(s_1(M),s_2(M),\tfrac{1}{2}\bigr). \]
	By Theorem~\ref{T:hom_class}, $M$ and $M'$ are orientation-preserving homotopy equivalent, since, by the first part of the proof, $\pi_4(M')=0=\pi_4(M)$.
	
	By \cite[Corollary 4.3]{EM16}, the invariants $2s_2$ and $s_3$ are oriented homotopy invariants. Hence, also $\tilde{s}_3=s_3-4s_2$ is an oriented homotopy invariant. It follows that $\tilde{s}_3(M)=\tilde{s}_3(M')$, which is a contradiction. Hence, $\pi_4(M)\cong\Z/2$.
	
	Alternatively, if $s_2(M)$ is of the form $\frac{j}{6}$, we can also use a similar proof as in the case $\tilde{s}_3=\frac{1}{2}$ by using the manifolds in Table \ref{TA:N_alpha} in combination with Lemmas \ref{L:princ_circ_hom} and \ref{L:s_inv_S2-bdls}.
\end{proof}

\section{Further results}\label{S:further}

First, we prove the following theorem, which directly implies Theorem~\ref{T:conn_sums}.
\begin{Theorem}\label{T:conn_sums_s}
	Let $M$ be a spin cohomology $S^2\times S^5$. If $M$ satisfies the hypotheses of $(1)$ of Theorem~$\ref{T:free_circle_s}$, then the manifolds
	\[ M\#_{(\ell-1)}\bigl(S^2\times S^5\bigr)\#_m\bigl(S^3\times S^4\bigr) \]
	for all $\ell\geq 2$, or $\ell=1$ and $m$ even, admit infinitely many pairwise non-equivalent free circle actions and for each of these action an invariant Riemannian metric of positive Ricci curvature.
\end{Theorem}

For the proof we need the following result.
\begin{Proposition}\label{P:tw_susp}
	Let $N$ be a spin manifold of dimension $6$ and let \smash{$M\xrightarrow{\pi} N$} be a principal circle bundle such that $M$ is simply-connected. Then there are principal circle bundles
	\[ M\#_{(\ell-1)}\bigl(S^2\times S^5\bigr)\#_m\bigl(S^3\times S^4\bigr)\to N\#_{(\ell-1)}\C P^3\#_{\frac{m}{2}}\bigl(S^3\times S^3\bigr) \]
	for all $\ell\geq 1$ and all $m\geq 0$ even, and
	\[ M\#_{(\ell-1)}\bigl(S^2\times S^5\bigr)\#_m\bigl(S^3\times S^4\bigr)\to N\#\bigl(S^2\times S^4\bigr)\#_{(\ell-2)}\C P^3\#_{\frac{m-1}{2}}\bigl(S^3\times S^3\bigr) \]
	for all $\ell\geq 2$ and all $m\geq 1$ odd.
\end{Proposition}
\begin{proof}
	This is essentially the content of \cite[Lemma~11]{Xu25}. For convenience, we give the proof below.
	
	Let $e_\pi\in H^2(N)$ be the Euler class of the principal circle bundle $\pi$ and let $a_i\in H^2\bigl(\C P^3\bigr)$ denote a generator for the $i$-th $\C P^3$-summand in \smash{$N_{\ell,m}=N\#_{(\ell-1)}\C P^3\#_{\frac{m}{2}}\bigl(S^3\times S^3\bigr)$}. We set
	\[ e_{\ell,m}=(e_\pi,a_1,\dots,a_{\ell-1})\in H^2(N_{\ell,m}). \]
	Let $\pi_{\ell,m}\colon M_{\ell,m}\to N_{\ell,m}$ be the unique principal circle bundle with Euler class $e_{\ell,m}$. Then, by \cite[Theorem A]{GR25}, we have a diffeomorphism
	\[ M_{\ell,m}\cong M\#\widetilde{\Sigma}_{a_1}\C P^3\#\dots\#\widetilde{\Sigma}_{a_{\ell-1}}\C P^3\#_{\frac{m}{2}}\widetilde{\Sigma}_0 \bigl(S^3\times S^3\bigr).  \]
	Here, $\widetilde{\Sigma}_e B$ for a manifold~$B$ and a class $e\in H^2(B)$ denotes the suspension operation of \cite{Du26} and~\cite{GR25}, that is, $\widetilde{\Sigma}_e B$ is the result of surgery on the total space of the principal circle bundle over~$B$ with Euler class $e$ along a fibre using the non-standard framing.
	
	Finally, by \cite[Theorem B]{GR25}, we have diffeomorphisms
	\[ \widetilde{\Sigma}_{a_i}\C P^3\cong S^2\times S^5,\qquad \widetilde{\Sigma}_0\bigl(S^3\times S^3\bigr)\cong \bigl(S^3\times S^4\bigr)\#\bigl(S^3\times S^4\bigr), \]
	which proves the first case. The proof of the second case is analogous.
\end{proof}

\begin{proof}[Proof of Theorem~\ref{T:conn_sums_s}]
	Suppose that $M$ satisfies the hypotheses of Theorem~\ref{T:free_circle_s}. Then, it was shown in the proof of Theorem~\ref{T:free_circle_s} that there are infinitely many principal circle bundles $\pi_i\colon M\to N_i$, $i\in\N$, where $N_i$ is a family of closed, simply-connected $6$-manifolds that are pairwise non-diffeomorphic. By Proposition \ref{P:tw_susp}, there are infinitely many principal circle bundles
	\[ M\#_{(\ell-1)}\bigl(S^2\times S^5\bigr)\#_m\bigl(S^3\times S^4\bigr)\to N_i\#_{(\ell-1)}\C P^3\#_{\frac{m}{2}}\bigl(S^3\times S^3\bigr) \]
	with $\ell\geq 1$ and $m\geq 0$ even, and{\samepage
	\[ M\#_{(\ell-1)}\bigl(S^2\times S^5\bigr)\#_m\bigl(S^3\times S^4\bigr)\to N_i\#(S^2\times S^4)\#_{(\ell-2)}\C P^3\#_{\frac{m-1}{2}}\bigl(S^3\times S^3\bigr) \]
	with $\ell\geq 2$ and $m\geq 1$ odd.}
	
	By an iterated application of Proposition \ref{P:splitting_diffeo}, for fixed $\ell$ and $m$, the families of manifolds
	\begin{gather*}
     N_i\#_{(\ell-1)}\C P^3\#_{\frac{m}{2}}\bigl(S^3\times S^3\bigr)\qquad \text{resp.}\  N_i\#\bigl(S^2\times S^4\bigr)\#_{(\ell-2)}\C P^3\#_{\frac{m-1}{2}}\bigl(S^3\times S^3\bigr),\qquad i\in\N
     \end{gather*}
	are pairwise non-diffeomorphic. Hence, the circle actions on the total spaces are all pairwise non-equivalent by Proposition \ref{P:action=bdl}.
	
	Finally, all the base spaces admit Riemannian metrics of positive Ricci curvature, and hence, by Theorem~\ref{T:Ric>0_bdl}, the total spaces admit invariant Riemannian metrics of positive Ricci curvature. Indeed, the manifolds $N_i$ in the proof of Theorem~\ref{T:free_circle_s} are of the form $N_A$, which admit \emph{core metrics} as defined in \cite{Bu19}, see \cite[Theorem B]{Re23}. Further, also the manifolds $\C P^3$, $S^2\times S^4$ and $S^3\times S^3$ admit core metrics by \cite[Theorem C]{Bu19}, \cite[Theorem C]{Re23} and \cite[Theorem B]{Bu20}, respectively. By \cite[Theorem B]{Bu19}, connected sums of manifolds with core metrics admit Riemannian metrics of positive Ricci curvature. Hence, all the base spaces admit Riemannian metrics of positive Ricci curvature.
\end{proof}

Finally, the following result illustrates why it is difficult in general to find suitable $6$-manifolds as base spaces for principal circle bundles with total space a spin cohomology $S^2\times S^5$.
\begin{Proposition}
	Let $M$ be a spin cohomology $S^2\times S^5$ and let $N$ be the orbit space of a free circle action on $M$. Suppose that $N$ splits non-trivially as a connected sum $N=N_1\# N_2$. Then both $N_1$ and $N_2$ are homotopy $\C P^3$s and $M$ is homeomorphic to $S^2\times S^5$.
\end{Proposition}
\begin{proof}
	By Lemma~\ref{L:N_orbit_space}, $N$ is simply-connected, has vanishing third Betti number and torsion-free cohomology. Hence, the same holds for $N_1$ and $N_2$. It follows that $N_1$ and $N_2$ satisfy $b_2(N_1)=b_2(N_2)=1$, since otherwise it would follow from Theorem~\ref{T:Jupp} that one of $N_1$ and $N_2$ is the $6$-sphere $S^6$.
	
	Let $e=\lambda_1 x_1+\lambda_2 x_2\in H^2(N_1\# N_2)\cong H^2(N)$ be the Euler class corresponding to the principal circle bundle $M\to N$. Since the number $AC-B^2$ in Lemma~ \ref{L:N_orbit_space} is the determinant of the bilinear form $(y,z)\mapsto \mu_N(y,z,e)$ in the basis $(e,f)$, we can equally compute it by considering the basis $(x_1,x_2)$ instead of $(e,f)$. Here we obtain from (3) of Lemma~\ref{L:inv_examples}
	\begin{align*}
		\pm 1=AC-B^2&=\det\begin{pmatrix}
			\mu_N(x_1,x_1,e) & \mu_N(x_1,x_2,e)\\
			\mu_N(x_2,x_1,e) & \mu_N(x_2,x_2,e)
		\end{pmatrix}\\
		&=\det\begin{pmatrix}
			\lambda_1 \mu_{N_1}(x_1,x_1,x_1) & 0\\
			0 & \lambda_2 \mu_{N_2}(x_2,x_2,x_2)
		\end{pmatrix}\\
		&=\lambda_1\lambda_2 \mu_{N_1}(x_1,x_1,x_1)\mu_{N_2}(x_2,x_2,x_2).
	\end{align*}
	Hence, all of $\lambda_1$, $\lambda_2$, $\mu_{N_1}(x_1,x_1,x_1)$, $\mu_{N_2}(x_2,x_2,x_2)$ are equal to $\pm1$.
	
	Now suppose that one of $N_i$, say $N_1$, is non-spin. Then, by (3) of Lemma~\ref{L:N_orbit_space}, $w_2(N_1)=x_1\mod 2$, so we have
	\[ a_1^3\mu_{N_1}(x_1,x_1,x_1)\equiv a_1 p_1(N_1)(x_1)\mod 48 \]
	for all odd $a_1\in\Z$ by \eqref{EQ:Jupp}. In particular, we have
	\[ 3p_1(N_1)(x_1)\equiv 27\mu_{N_1}(x_1,x_1,x_1)\equiv 27p_1(N_1)(x_1)\mod 48, \]
	so $24p_1(N_1)(x_1)\equiv 0\mod 48$, thus $p_1(N_1)(x_1)$ is even. It follows that $\mu_{N_1}(x_1,x_1,x_1)\equiv p_1(N_1)(x_1)\mod 48$ is even as well, which is a contradiction. Hence, both $N_1$ and $N_2$ are spin.\looseness=-1
	
	Since $\mu_{N_1}(x_1,x_1,x_1)=\pm1$, it then follows that both $N_i$ have the cohomology ring of $\C P^3$ and hence are homotopy $\C P^3$s (see, e.g., \cite[p.\ 41]{Re24b} and (2) of Lemma~\ref{L:inv_examples}).
	
	Finally, since the restriction of $e$ to each $N_i$ is primitive as $\lambda_i=\pm1$, it follows from \cite[Theorems A and B]{GR25} that
$ M\cong M_1\# M_2\# \bigl(S^2\times S^5\bigr)$,
	where $M_i$ is the total space of the principal circle bundle over $N_i$ with Euler class $\lambda_i x_i$. Each $M_i$ is a homotopy $7$-sphere by \cite[Lemma~6]{MY66}, so $M$ is of the form $\bigl(S^2\times S^5\bigr)\#\Sigma$ with $\Sigma $ a homotopy $7$-sphere. Hence, $M$ is homeomorphic to $S^2\times S^5$.
\end{proof}

\appendix

\section{Tables}\label{S:Tables}

{
	
	\scriptsize
	\setlength{\tabcolsep}{0.6em}
	
	\centering \renewcommand{\arraystretch}{0.98}
	
	\begin{paracol}{3}
		
\begin{supertabular}{|w{c}{.32cm}w{c}{.32cm}w{c}{.32cm}@{\hspace{.5cm}}w{c}{.2cm}w{c}{.2cm}|@{\hspace{\tabcolsep}}HHHHHHw{c}{.17cm}w{c}{.17cm}w{c}{.17cm}|}
			\hline $\alpha_1$ & $\alpha_2$ & $\alpha_3$ & $\lambda$ & $\mu$ & $A$ & $B$ & $C$ & $D$ & $u$ & $v$ & $28s_1$ & $12s_2$ & $2\tilde{s}_3$ \\ \hline
			0 & $-$1 & $-$1 & 0 & $-$1 & 0 & 1 & $-$1 & 1 & 0 & 0 & 0 & 0 & 0 \\
			19 & 13 & $-$4 & 1 & $-$7 & $-$5304 & 781 & $-$115 & 17 & 868 & 0 & 1 & 0 & 0 \\
			47 & $-$7 & $-$4 & 1 & $-$7 & 1584 & $-$199 & 25 & $-$3 & $-$252 & 0 & 2 & 0 & 0 \\
			23 & 2 & 5 & 1 & $-$6 & 216 & $-$53 & 13 & $-$3 & $-$30 & 0 & 3 & 0 & 0 \\
			43 & 33 & $-$38 & 1 & $-$3 & $-$1152 & 449 & $-$175 & 71 & 170 & 0 & 4 & 0 & 0 \\
			19 & 3 & $-$8 & 1 & $-$9 & $-$6264 & 755 & $-$91 & 11 & 1032 & 0 & 5 & 0 & 0 \\
			23 & 2 & 17 & 1 & $-$10 & 10416 & $-$1177 & 133 & $-$15 & $-$1710 & 0 & 6 & 0 & 0 \\
\hline
		\end{supertabular}
		
		\switchcolumn
		
		\begin{supertabular}{|w{c}{.32cm}w{c}{.32cm}w{c}{.32cm}@{\hspace{.5cm}}w{c}{.2cm}w{c}{.2cm}|@{\hspace{\tabcolsep}}HHHHHHw{c}{.17cm}w{c}{.17cm}w{c}{.17cm}|}
			\hline $\alpha_1$ & $\alpha_2$ & $\alpha_3$ & $\lambda$ & $\mu$ & $A$ & $B$ & $C$ & $D$ & $u$ & $v$ & $28s_1$ & $12s_2$ & $2\tilde{s}_3$ \\ \hline
			$-$1 & $-$1 & $-$1 & 0 & $-$1 & 0 & 1 & $-$1 & 0 & 0 & 0 & 0 & 0 & 1 \\
			19 & 3 & $-$13 & 1 & $-$8 & $-$5976 & 829 & $-$115 & 16 & 980 & 0 & 1 & 0 & 1 \\
			47 & 3 & 7 & 1 & $-$8 & 912 & $-$151 & 25 & $-$4 & $-$140 & 0 & 2 & 0 & 1 \\
			41 & $-$23 & $-$221 & 3 & 5 & 111384 & 41651 & 15575 & 5824 & $-$18268 & $-$861 & 3 & 0 & 1 \\
			29 & 11 & $-$1531 & 3 & $-$8 & $-$196224 & 78401 & $-$31325 & 12516 & 31428 & $-$1575 & 4 & 0 & 1 \\
			43 & 31 & 3199 & 3 & $-$4 & 2376 & $-$109 & 5 & 12 & 138 & 0 & 5 & 0 & 1 \\
			11 & $-$1 & $-$1 & 1 & $-$4 & 48 & $-$7 & 1 & 0 & $-$6 & 0 & 6 & 0 & 1 \\
\hline
		\end{supertabular}
		
		\switchcolumn
		
		\begin{supertabular}{|w{c}{.32cm}w{c}{.32cm}w{c}{.32cm}@{\hspace{.5cm}}w{c}{.2cm}w{c}{.2cm}|@{\hspace{\tabcolsep}}HHHHHHw{c}{.17cm}w{c}{.17cm}w{c}{.17cm}|}
			\hline $\alpha_1$ & $\alpha_2$ & $\alpha_3$ & $\lambda$ & $\mu$ & $A$ & $B$ & $C$ & $D$ & $u$ & $v$ & $28s_1$ & $12s_2$ & $2\tilde{s}_3$ \\ \hline
			19 & 5 & 21 & 1 & $-$2 & 0 & $-$1 & 11 & $-$16 & 5 & 0 & 0 & 6 & 1 \\
			47 & 5 & $-$87 & 1 & $-$10 & $-$68376 & 7547 & $-$833 & 92 & 11265 & 0 & 1 & 6 & 1 \\
			37 & $-$5 & $-$3 & 1 & $-$7 & 1104 & $-$137 & 17 & $-$2 & $-$175 & 0 & 2 & 6 & 1 \\
			1 & $-$27 & $-$1 & 1 & $-$2 & 216 & $-$107 & 53 & $-$26 & $-$27 & 0 & 3 & 6 & 1 \\
			19 & 9 & 341 & 1 & $-$2 & 288 & $-$305 & 323 & $-$332 & 9 & 0 & 4 & 6 & 1 \\
			37 & 3 & 7 & 1 & $-$6 & 264 & $-$67 & 17 & $-$4 & $-$35 & 0 & 5 & 6 & 1 \\
			29 & 23 & $-$39 & 2 & $-$5 & $-$3696 & $-$1385 & $-$519 & $-$194 & 587 & 23 & 6 & 6 & 1 \\
\hline	
		\end{supertabular}
	\end{paracol}
	\vspace{-.2cm}
	\captionof{table}{List of the $252$ spin cohomology $S^2\times S^5$s of the form $N_A(e)$ with $\det(N_A,e)=-1$, sorted by their $s$-invariants}
	\label{TA:det=-1}
	\vspace{.5cm}
	
	\begin{paracol}{3}
		\begin{supertabular}{|w{c}{.32cm}w{c}{.32cm}w{c}{.32cm}@{\hspace{.5cm}}w{c}{.2cm}w{c}{.2cm}|@{\hspace{\tabcolsep}}HHHHHHw{c}{.17cm}w{c}{.17cm}w{c}{.17cm}|}
			\hline $\alpha_1$ & $\alpha_2$ & $\alpha_3$ & $\lambda$ & $\mu$ & $A$ & $B$ & $C$ & $D$ & $u$ & $v$ & $28s_1$ & $12s_2$ & $2\tilde{s}_3$ \\ \hline
			0 & $-$1 & 1 & 0 & $-$1 & 2 & $-$1 & 1 & $-$1 & 0 & 0 & 0 & 0 & 0 \\
			23 & 5 & $-$72 & 1 & $-$5 & $-$5210 & 1277 & $-$313 & 77 & 820 & 0 & 1 & 0 & 0 \\
			25 & $-$12 & $-$1 & 1 & $-$24 & 153746 & $-$6383 & 265 & $-$11 & $-$25576 & 0 & 2 & 0 & 0 \\
			7 & 8 & $-$3 & 1 & $-$4 & $-$586 & 155 & $-$41 & 11 & 92 & 0 & 3 & 0 & 0 \\
			12 & $-$1 & 1 & 0 & $-$1 & 2 & $-$1 & 1 & 11 & 0 & 0 & 4 & 0 & 0 \\
			3 & 1 & $-$2 & 1 & $-$5 & $-$250 & 57 & $-$13 & 3 & 40 & 0 & 5 & 0 & 0 \\
			23 & 12 & $-$553 & 1 & $-$2 & $-$626 & 601 & $-$577 & 565 & 12 & 0 & 6 & 0 & 0 \\
\hline
		\end{supertabular}
		
		\switchcolumn
		
		\begin{supertabular}{|w{c}{.32cm}w{c}{.32cm}w{c}{.32cm}@{\hspace{.5cm}}w{c}{.2cm}w{c}{.2cm}|@{\hspace{\tabcolsep}}HHHHHHw{c}{.17cm}w{c}{.17cm}w{c}{.17cm}|}
			\hline $\alpha_1$ & $\alpha_2$ & $\alpha_3$ & $\lambda$ & $\mu$ & $A$ & $B$ & $C$ & $D$ & $u$ & $v$ & $28s_1$ & $12s_2$ & $2\tilde{s}_3$ \\ \hline
			31 & 19 & 707 & 2 & $-$3 & 442 & $-$47 & 5 & 12 & 45 & 0 & 0 & 1 & 1 \\
			$-$1 & $-$1 & 157 & 13 & $-$12 & $-$626 & 25 & $-$1 & 0 & 78 & 0 & 1 & 1 & 1 \\
			7 & 11 & $-$289 & 5 & $-$3 & 2890 & 1407 & 685 & 334 & $-$385 & $-$7 & 2 & 1 & 1 \\
			5 & 1 & $-$1 & 0 & $-$1 & $-$2 & 1 & $-$1 & 6 & 0 & 0 & 3 & 1 & 1 \\
			47 & $-$41 & $-$17 & 1 & $-$3 & 1018 & $-$301 & 89 & $-$24 & $-$147 & 0 & 4 & 1 & 1 \\
			25 & 19 & $-$43 & 3 & $-$7 & $-$8594 & $-$2325 & $-$629 & $-$170 & 1394 & 19 & 5 & 1 & 1 \\
			31 & 19 & 431 & 5 & $-$6 & 202 & 91 & 41 & 12 & 45 & 0 & 6 & 1 & 1 \\
\hline
		\end{supertabular}
		
		\switchcolumn
		
		\begin{supertabular}{|w{c}{.32cm}w{c}{.32cm}w{c}{.32cm}@{\hspace{.5cm}}w{c}{.2cm}w{c}{.2cm}|@{\hspace{\tabcolsep}}HHHHHHw{c}{.17cm}w{c}{.17cm}w{c}{.17cm}|}
			\hline $\alpha_1$ & $\alpha_2$ & $\alpha_3$ & $\lambda$ & $\mu$ & $A$ & $B$ & $C$ & $D$ & $u$ & $v$ & $28s_1$ & $12s_2$ & $2\tilde{s}_3$ \\ \hline
			$-$1 & $-$1 & 7 & 3 & $-$2 & $-$26 & 5 & $-$1 & 0 & 3 & 0 & 0 & 7 & 1 \\
			19 & $-$43 & 109 & 2 & 1 & $-$2834 & 905 & $-$289 & 90 & 417 & 0 & 1 & 7 & 1 \\
			19 & 263 & 47 & 62 & $-$33 & $-$6069482 & 1472405 & $-$357193 & 86652 & 1010103 & $-$14084 & 2 & 7 & 1 \\
			11 & 5 & 361 & 7 & $-$15 & 171730 & 21393 & 2665 & 332 & $-$28140 & 5 & 3 & 7 & 1 \\
			13 & 17 & 2269 & 14 & $-$11 & $-$48218 & 15909 & $-$5249 & 1732 & 6901 & 90 & 4 & 7 & 1 \\
			11 & 13 & 6007 & 7 & $-$6 & $-$5042 & $-$71 & $-$1 & 2 & $-$161 & 0 & 5 & 7 & 1 \\
			$-$1 & $-$1 & 43 & 7 & $-$6 & $-$170 & 13 & $-$1 & 0 & 21 & 0 & 6 & 7 & 1 \\
\hline
		\end{supertabular}
	\end{paracol}
	\vspace{-.2cm}
	\captionof{table}{List of the $252$ spin cohomology $S^2\times S^5$s of the form $N_A(e)$ with $\det(N_A,e)=1$, sorted by their $s$-invariants}
	\label{TA:det=1}
	\vspace{.5cm}
	
	\begin{paracol}{4}
		\begin{supertabular}{|w{c}{.35cm}w{c}{.27cm}w{c}{.25cm}|@{\hspace{\tabcolsep}}HHHHHHw{c}{.17cm}w{c}{.17cm}w{c}{.17cm}|}
			\hline $\alpha$ & $\lambda$ & $\mu$ & $A$ & $B$ & $C$ & $D$ & $u$ & $v$ & $28s_1$ & $12s_2$ & $2\tilde{s}_3$ \\ \hline
			0 & 1 & 0 & 0 & 1 & 2 & 3 & 0 & $-$1 & 0 & 0 & 0 \\
			333 & 73 & $-$4 & $-$85260 & 21169 & $-$5256 & 1305 & 1665 & $-$2160 & 2 & 0 & 0 \\
			132 & 23 & 2 & 4230 & 2069 & 1012 & 495 & $-$66 & $-$440 & 4 & 0 & 0 \\
			3 & 7 & $-$4 & $-$780 & $-$209 & $-$56 & $-$15 & 15 & 36 & 6 & 0 & 0 \\
			12 & 7 & 2 & 390 & 181 & 84 & 39 & $-$6 & $-$36 & 8 & 0 & 0 \\
			33 & 23 & $-$4 & $-$8460 & $-$2161 & $-$552 & $-$141 & 165 & 360 & 10 & 0 & 0 \\
			72 & 17 & 2 & 2310 & 1121 & 544 & 264 & $-$36 & $-$236 & 12 & 0 & 0 \\
			105 & 41 & $-$4 & $-$26892 & 6641 & $-$1640 & 405 & 525 & $-$680 & 14 & 0 & 0 \\
			18 & 17 & $-$4 & $-$4620 & 1121 & $-$272 & 66 & 90 & $-$116 & 16 & 0 & 0 \\
			495 & 89 & $-$4 & $-$126732 & 31505 & $-$7832 & 1947 & 2475 & $-$3212 & 18 & 0 & 0 \\
			12 & 7 & $-$2 & $-$390 & 181 & $-$84 & 39 & 6 & $-$18 & 20 & 0 & 0 \\
			3 & 7 & 4 & 780 & $-$209 & 56 & $-$15 & $-$15 & 20 & 22 & 0 & 0 \\
			2 & 17 & 12 & 13860 & 5741 & 2378 & 985 & $-$286 & $-$1174 & 24 & 0 & 0 \\
			39 & 25 & $-$4 & $-$9996 & 2449 & $-$600 & 147 & 195 & $-$252 & 26 & 0 & 0 \\
			3782 & 123 & $-$2 & $-$121030 & 60269 & $-$30012 & 14945 & 1891 & $-$5673 & 2 & 2 & 0 \\
			110 & 21 & $-$2 & $-$3526 & 1721 & $-$840 & 410 & 55 & $-$165 & 6 & 2 & 0 \\\hline
		\end{supertabular}
		
		\switchcolumn
		
		\begin{supertabular}{|w{c}{.35cm}w{c}{.27cm}w{c}{.25cm}|@{\hspace{\tabcolsep}}HHHHHHw{c}{.17cm}w{c}{.17cm}w{c}{.17cm}|}
			\hline $\alpha$ & $\lambda$ & $\mu$ & $A$ & $B$ & $C$ & $D$ & $u$ & $v$ & $28s_1$ & $12s_2$ & $2\tilde{s}_3$ \\ \hline
			26 & 51 & 10 & 104030 & 10301 & 1020 & 101 & $-$2145 & $-$1425 & 10 & 2 & 0 \\
			182 & 27 & $-$2 & $-$5830 & 2861 & $-$1404 & 689 & 91 & $-$273 & 14 & 2 & 0 \\
			122 & 243 & 22 & 5196290 & 235709 & 10692 & 485 & $-$108031 & $-$30855 & 18 & 2 & 0 \\
			2 & 3 & $-$2 & $-$70 & 29 & $-$12 & 5 & 1 & $-$3 & 22 & 2 & 0 \\
			170 & 339 & $-$26 & $-$11951758 & 459005 & $-$17628 & 677 & 248625 & $-$55263 & 26 & 2 & 0 \\
			14 & 15 & $-$4 & $-$3596 & $-$929 & $-$240 & $-$62 & 70 & 156 & 0 & 4 & 0 \\
			473 & 87 & $-$4 & $-$121100 & $-$30449 & $-$7656 & $-$1925 & 2365 & 5016 & 2 & 4 & 0 \\
			20 & 9 & 2 & 646 & 305 & 144 & 68 & $-$10 & $-$62 & 4 & 4 & 0 \\
			689 & 105 & $-$4 & $-$176396 & 43889 & $-$10920 & 2717 & 3445 & $-$4472 & 6 & 4 & 0 \\
			992 & 63 & 2 & 31750 & 15749 & 7812 & 3875 & $-$496 & $-$3410 & 8 & 4 & 0 \\
			5 & 9 & $-$4 & $-$1292 & 305 & $-$72 & 17 & 25 & $-$32 & 10 & 4 & 0 \\
			410 & 81 & $-$4 & $-$104972 & 26081 & $-$6480 & 1610 & 2050 & $-$2660 & 12 & 4 & 0 \\
			203 & 57 & $-$4 & $-$51980 & 12881 & $-$3192 & 791 & 1015 & $-$1316 & 14 & 4 & 0 \\
			1040 & 129 & $-$4 & $-$266252 & 66305 & $-$16512 & 4112 & 5200 & $-$6752 & 16 & 4 & 0 \\
			2525 & 201 & $-$4 & $-$646412 & 161201 & $-$40200 & 10025 & 12625 & $-$16400 & 18 & 4 & 0 \\
			380 & 39 & 2 & 12166 & 6005 & 2964 & 1463 & $-$190 & $-$1292 & 20 & 4 & 0 \\\hline
		\end{supertabular}
		
		\switchcolumn
		
		\begin{supertabular}{|w{c}{.35cm}w{c}{.27cm}w{c}{.25cm}|@{\hspace{\tabcolsep}}HHHHHHw{c}{.17cm}w{c}{.17cm}w{c}{.17cm}|}
			\hline $\alpha$ & $\lambda$ & $\mu$ & $A$ & $B$ & $C$ & $D$ & $u$ & $v$ & $28s_1$ & $12s_2$ & $2\tilde{s}_3$ \\ \hline
			1055 & 1689 & $-$52 & $-$593365916 & 285268697 & $-$137146788 & 65935175 & 12357215 & $-$29500437 & 22 & 4 & 0 \\
			1958 & 177 & $-$4 & $-$501260 & 124961 & $-$31152 & 7766 & 9790 & $-$12716 & 24 & 4 & 0 \\
			95 & 39 & $-$4 & $-$24332 & $-$6161 & $-$1560 & $-$395 & 475 & 1020 & 26 & 4 & 0 \\
			30 & 11 & $-$2 & $-$966 & 461 & $-$220 & 105 & 15 & $-$45 & 2 & 6 & 0 \\
			30 & 11 & 2 & 966 & 461 & 220 & 105 & $-$15 & $-$95 & 6 & 6 & 0 \\
			6 & 5 & $-$2 & $-$198 & 89 & $-$40 & 18 & 3 & $-$9 & 10 & 6 & 0 \\
			6 & 5 & 2 & 198 & 89 & 40 & 18 & $-$3 & $-$17 & 14 & 6 & 0 \\
			10 & 19 & $-$6 & $-$8658 & 1405 & $-$228 & 37 & 175 & $-$153 & 18 & 6 & 0 \\
			306 & 35 & 2 & 9798 & 4829 & 2380 & 1173 & $-$153 & $-$1037 & 22 & 6 & 0 \\
			870 & 59 & 2 & 27846 & 13805 & 6844 & 3393 & $-$435 & $-$2987 & 26 & 6 & 0 \\
			20 & 9 & $-$2 & $-$646 & 305 & $-$144 & 68 & 10 & $-$30 & 0 & 8 & 0 \\
			473 & 87 & 4 & 121100 & $-$30449 & 7656 & $-$1925 & $-$2365 & 3080 & 2 & 8 & 0 \\
			68 & 33 & 4 & 17420 & 4289 & 1056 & 260 & $-$340 & $-$696 & 4 & 8 & 0 \\
			11567 & 2151 & $-$20 & $-$370144060 & $-$166569129 & $-$74958044 & $-$33731991 & 7692055 & 35796029 & 6 & 8 & 0 \\
			248 & 63 & 4 & 63500 & $-$16001 & 4032 & $-$1016 & $-$1240 & 1616 & 8 & 8 & 0 \\
			5 & 9 & 4 & 1292 & 305 & 72 & 17 & $-$25 & $-$48 & 10 & 8 & 0 \\\hline
		\end{supertabular}
		
		\switchcolumn
		
		\begin{supertabular}{|w{c}{.35cm}w{c}{.27cm}w{c}{.25cm}|@{\hspace{\tabcolsep}}HHHHHHw{c}{.17cm}w{c}{.17cm}w{c}{.17cm}|}
			\hline $\alpha$ & $\lambda$ & $\mu$ & $A$ & $B$ & $C$ & $D$ & $u$ & $v$ & $28s_1$ & $12s_2$ & $2\tilde{s}_3$ \\ \hline
			410 & 81 & 4 & 104972 & 26081 & 6480 & 1610 & $-$2050 & $-$4260 & 12 & 8 & 0 \\
			203 & 57 & 4 & 51980 & 12881 & 3192 & 791 & $-$1015 & $-$2100 & 14 & 8 & 0 \\
			380 & 39 & $-$2 & $-$12166 & 6005 & $-$2964 & 1463 & 190 & $-$570 & 16 & 8 & 0 \\
			101 & 201 & $-$20 & $-$3232060 & 161201 & $-$8040 & 401 & 67165 & $-$19200 & 18 & 8 & 0 \\
			992 & 63 & $-$2 & $-$31750 & 15749 & $-$7812 & 3875 & 496 & $-$1488 & 20 & 8 & 0 \\
			689 & 105 & 4 & 176396 & 43889 & 10920 & 2717 & $-$3445 & $-$7176 & 22 & 8 & 0 \\
			272 & 33 & $-$2 & $-$8710 & 4289 & $-$2112 & 1040 & 136 & $-$408 & 24 & 8 & 0 \\
			95 & 39 & 4 & 24332 & $-$6161 & 1560 & $-$395 & $-$475 & 620 & 26 & 8 & 0 \\
			3014 & 549 & $-$10 & $-$12056030 & $-$1206701 & $-$120780 & $-$12089 & 248655 & 167475 & 2 & 10 & 0 \\
			3782 & 123 & 2 & 121030 & 60269 & 30012 & 14945 & $-$1891 & $-$13115 & 6 & 10 & 0 \\
			26 & 51 & $-$10 & $-$104030 & 10301 & $-$1020 & 101 & 2145 & $-$1175 & 10 & 10 & 0 \\
			182 & 27 & 2 & 5830 & 2861 & 1404 & 689 & $-$91 & $-$611 & 14 & 10 & 0 \\
			650 & 51 & 2 & 20806 & 10301 & 5100 & 2525 & $-$325 & $-$2225 & 18 & 10 & 0 \\
			110 & 21 & 2 & 3526 & 1721 & 840 & 410 & $-$55 & $-$365 & 22 & 10 & 0 \\
			2 & 3 & 2 & 70 & 29 & 12 & 5 & $-$1 & $-$5 & 26 & 10 & 0 \\ \hline
		\end{supertabular}
	\end{paracol}
	\vspace{-.2cm}
	\captionof{table}{List of the $63$ spin cohomology $S^2\times S^5$s of the form $N_{(\alpha,0)}(e)$, sorted by their $s$-invariants}
	\label{TA:N_alpha}
}

\subsection*{Acknowledgements}
	The author would like to thank Ruizhi Huang for helpful conversations on his article \cite{Hu25} and the anonymous referees for their thorough reading and helpful suggestions.

\pdfbookmark[1]{References}{ref}
\LastPageEnding


\begin{thebibliography}{99}
\footnotesize\itemsep=0pt

\bibitem{AW14}
Antieau B., Williams B., The topological period-index problem over 6-complexes,
  \href{https://doi.org/10.1112/jtopol/jtt042}{\textit{J.~Topol.}} \textbf{7}
  (2014), 617--640, \href{https://arxiv.org/abs/1208.4430}{arXiv:1208.4430}.

\bibitem{BQ26}
Bauer T., Quigley J.D., Infinite families of very exotic spheres with
  free~${S}^1$- and~${S}^3$-actions,
  \href{https://arxiv.org/abs/2603.23241}{arXiv:2603.23241}.

\bibitem{BB78}
B\'erard-Bergery L., Certains fibr\'es \`a courbure de {R}icci positive,
  \textit{C.~R.~Acad. Sci. Paris S\'er. A-B} \textbf{286} (1978), A929--A931.

\bibitem{BB83}
B\'erard-Bergery L., Scalar curvature and isometry group, in Spectra of
  Riemannian Manifolds, Kaigai Publications, Tokyo, 1983, 9--28.

\bibitem{Br72a}
Bredon G.E., Introduction to compact transformation groups, \textit{Pure Appl.
  Math.}, Vol.~46, Academic Press, New York, 1972.

\bibitem{Bu19a}
Burdick B.L., Metrics of positive {R}icci curvature on connected sums:
  projective spaces, products, and plumbings, Ph.D.~Thesis, {U}niversity of
  Oregon, 2019, available at \url{https://www.proquest.com/docview/2293989895}.

\bibitem{Bu19}
Burdick B.L., Ricci-positive metrics on connected sums of projective spaces,
  \href{https://doi.org/10.1016/j.difgeo.2018.11.005}{\textit{Differential
  Geom. Appl.}} \textbf{62} (2019), 212--233,
  \href{https://arxiv.org/abs/1705.05055}{arXiv:1705.05055}.

\bibitem{Bu20}
Burdick B.L., Metrics of positive {R}icci curvature on the connected sums of
  products with arbitrarily many spheres,
  \href{https://doi.org/10.1007/s10455-020-09732-7}{\textit{Ann. Global Anal.
  Geom.}} \textbf{58} (2020), 433--476,
  \href{https://arxiv.org/abs/1811.11906}{arXiv:1811.11906}.

\bibitem{DK01}
Davis J.F., Kirk P., Lecture notes in algebraic topology, \textit{Grad. Stud.
  Math.}, Vol.~35, \href{https://doi.org/10.1090/gsm/035}{American Mathematical
  Society}, Providence, RI, 2001.

\bibitem{Du26}
Duan H., Circle actions and suspension operations on smooth manifolds,
  \href{https://doi.org/10.1017/S0305004125101849}{\textit{Math. Proc.
  Cambridge Philos. Soc.}} \textbf{180} (2026), 643--662,
  \href{https://arxiv.org/abs/2202.06225}{arXiv:2202.06225}.

\bibitem{DL05}
Duan H., Liang C., Circle bundles over 4-manifolds,
  \href{https://doi.org/10.1007/s00013-005-1214-4}{\textit{Arch. Math.
  (Basel)}} \textbf{85} (2005), 278--282,
  \href{https://arxiv.org/abs/math/0502469}{arXiv:math/0502469}.

\bibitem{EK62}
Eells Jr. J., Kuiper N.H., An~invariant for certain smooth manifolds,
  \href{https://doi.org/10.1007/BF02412768}{\textit{Ann. Mat. Pura Appl.}}
  \textbf{60} (1962), 93--110.

\bibitem{EZ14}
Escher C., Ziller W., Topology of non-negatively curved manifolds,
  \href{https://doi.org/10.1007/s10455-013-9407-8}{\textit{Ann. Global Anal.
  Geom.}} \textbf{46} (2014), 23--55,
  \href{https://arxiv.org/abs/1206.5997}{arXiv:1206.5997}.

\bibitem{EM16}
Escher C.M., Montagantirud P., Classifying seven dimensional manifolds of fixed
  cohomology type,
  \href{https://doi.org/10.1016/j.difgeo.2016.09.006}{\textit{Differential
  Geom. Appl.}} \textbf{49} (2016), 312--325.

\bibitem{GR25}
Galaz-Garc\'{\i}a F., Reiser P., Free torus actions and twisted suspensions,
  \href{https://doi.org/10.1017/fms.2024.141}{\textit{Forum Math. Sigma}}
  \textbf{13} (2025), e3, 31~pages,
  \href{https://arxiv.org/abs/2305.06068}{arXiv:2305.06068}.

\bibitem{GPT98}
Gilkey P.B., Park J., Tuschmann W., Invariant metrics of positive {R}icci
  curvature on principal bundles,
  \href{https://doi.org/10.1007/PL00004385}{\textit{Math.~Z.}} \textbf{227}
  (1998), 455--463.

\bibitem{GL71}
Goldstein R.Z., Lininger L., A~classification of {$6$}-manifolds with free
  {$S^1$} actions, in Proceedings of the {S}econd {C}onference on {C}ompact
  {T}ransformation {G}roups ({U}niv. {M}assachusetts, {A}mherst, {M}ass.,
  1971), {P}art {I}, \textit{Lecture Notes in Math.}, Vol.~298,
  \href{https://doi.org/10.1007/BFb0070052}{Springer}, Berlin, 1972, 316--323.

\bibitem{GW09}
Gromoll D., Walschap G., Metric foliations and curvature, \textit{Progr.
  Math.}, Vol. 268,
  \href{https://doi.org/10.1007/978-3-7643-8715-0}{Birkh\"auser}, Basel, 2009.

\bibitem{Hs66}
Hsiang W.-C., A~note on free differentiable actions of~{$S^1$} and~{$S^{3}$} on
  homotopy spheres, \href{https://doi.org/10.2307/1970431}{\textit{Ann. of
  Math.}} \textbf{83} (1966), 266--272.

\bibitem{Hu25}
Huang R., Sphere bundles over 4-manifolds are trivial after looping,
  \href{https://doi.org/10.1007/s00209-025-03757-z}{\textit{Math.~Z.}}
  \textbf{310} (2025), 48, 12~pages,
  \href{https://arxiv.org/abs/2210.17352}{arXiv:2210.17352}.

\bibitem{Hu94}
Husemoller D., Fibre bundles, 3rd ed., \textit{Grad. Texts in Math.}, Vol.~20,
  \href{https://doi.org/10.1007/978-1-4757-2261-1}{Springer}, New York, 1994.

\bibitem{Ji14}
Jiang Y., Regular circle actions on 2-connected 7-manifolds,
  \href{https://doi.org/10.1112/jlms/jdu028}{\textit{J.~Lond. Math. Soc.}}
  \textbf{90} (2014), 373--387,
  \href{https://arxiv.org/abs/1302.0923}{arXiv:1302.0923}.

\bibitem{JS25}
Jiang Y., Su Y., Free circle actions on {$(n-1)$}-connected
  {$(2n+1)$}-manifolds,
  \href{https://doi.org/10.2140/pjm.2025.338.1}{\textit{Pacific~J. Math.}}
  \textbf{338} (2025), 1--17,
  \href{https://arxiv.org/abs/2409.03194}{arXiv:2409.03194}.

\bibitem{Ju73}
Jupp P.E., Classification of certain {$6$}-manifolds,
  \href{https://doi.org/10.1017/s0305004100076854}{\textit{Proc. Cambridge
  Philos. Soc.}} \textbf{73} (1973), 293--300.

\bibitem{Ko58}
Kobayashi S., Fixed points of isometries,
  \href{https://doi.org/10.1017/S0027763000023497}{\textit{Nagoya Math.~J.}}
  \textbf{13} (1958), 63--68.

\bibitem{KS88}
Kreck M., Stolz S., A~diffeomorphism classification of {$7$}-dimensional
  homogeneous {E}instein manifolds with {${\rm SU}(3)\times{\rm
  SU}(2)\times{\rm U}(1)$}-symmetry,
  \href{https://doi.org/10.2307/2007059}{\textit{Ann. of Math.}} \textbf{127}
  (1988), 373--388.

\bibitem{KS91}
Kreck M., Stolz S., Some nondiffeomorphic homeomorphic homogeneous
  {$7$}-manifolds with positive sectional curvature,
  \href{https://doi.org/10.4310/jdg/1214446327}{\textit{J.~Differential Geom.}}
  \textbf{33} (1991), 465--486.

\bibitem{KS98}
Kreck M., Stolz S., A~correction on ``{S}ome nondiffeomorphic homeomorphic
  homogeneous {$7$}-manifolds with positive sectional curvature'',
  \href{https://doi.org/10.4310/jdg/1214460941}{\textit{J.~Differential Geom.}}
  \textbf{49} (1998), 203--204.

\bibitem{KS25}
Kreck M., Su Y., Mapping class group of manifolds which look like 3-dimensional
  complete intersections,
  \href{https://doi.org/10.1215/00127094-2024-0036}{\textit{Duke Math.~J.}}
  \textbf{174} (2025), 501--574,
  \href{https://arxiv.org/abs/2009.08054}{arXiv:2009.08054}.

\bibitem{Kr97}
Kruggel B., A~homotopy classification of certain {$7$}-manifolds,
  \href{https://doi.org/10.1090/S0002-9947-97-01962-4}{\textit{Trans. Amer.
  Math. Soc.}} \textbf{349} (1997), 2827--2843.

\bibitem{Kr98}
Kruggel B., Kreck--{S}tolz invariants, normal invariants and the homotopy
  classification of generalised {W}allach spaces,
  \href{https://doi.org/10.1093/qmathj/49.4.469}{\textit{Quart.~J.~Math. Oxford
  Ser.~(2)}} \textbf{49} (1998), 469--485.

\bibitem{MS74}
Milnor J.W., Stasheff J.D., Characteristic classes, \textit{Ann. of Math.
  Stud.}, Vol.~76, \href{https://doi.org/10.1515/9781400881826}{Princeton
  University Press}, Princeton, NJ, 1974.

\bibitem{MY66}
Montgomery D., Yang C.T., Differentiable actions on homotopy seven spheres,
  \href{https://doi.org/10.2307/1994562}{\textit{Trans. Amer. Math. Soc.}}
  \textbf{122} (1966), 480--498.

\bibitem{MY68}
Montgomery D., Yang C.T., Differentiable actions on homotopy seven
  spheres.~{II}, in Proc. {C}onf. on {T}ransformation {G}roups ({N}ew
  {O}rleans, {L}a., 1967),
  \href{https://doi.org/10.1007/978-3-642-46141-5_4}{Springer}, New York, 1968,
  125--134.

\bibitem{Re22}
Reiser P., Generalized surgery on {R}iemannian manifolds of positive {R}icci
  curvature, Ph.D.~Thesis, {K}arlsruher Institut f\"ur Technologie, 2024,
  available at \url{https://doi.org/10.5445/IR/1000150280}.

\bibitem{Re23}
Reiser P., Generalized surgery on {R}iemannian manifolds of positive {R}icci
  curvature, \href{https://doi.org/10.1090/tran/8789}{\textit{Trans. Amer.
  Math. Soc.}} \textbf{376} (2023), 3397--3418,
  \href{https://arxiv.org/abs/2103.05517}{arXiv:2103.05517}.

\bibitem{Re24b}
Reiser P., Metrics of positive {R}icci curvature on simply-connected manifolds
  of dimension~{$6k$},
  \href{https://doi.org/10.1112/topo.70007}{\textit{J.~Topol.}} \textbf{17}
  (2024), e70007, 50~pages,
  \href{https://arxiv.org/abs/2210.15610}{arXiv:2210.15610}.

\bibitem{Sc71}
Schultz R., The nonexistence of free~{$S^{1}$} actions on some homotopy
  spheres, \href{https://doi.org/10.2307/2036505}{\textit{Proc. Amer. Math.
  Soc.}} \textbf{27} (1971), 595--597.

\bibitem{Th74}
Thomas A., Almost complex structures on complex projective spaces,
  \href{https://doi.org/10.2307/1996904}{\textit{Trans. Amer. Math. Soc.}}
  \textbf{193} (1974), 123--132.

\bibitem{Wa66}
Wall C.T.C., Classification problems in differential topology. {V}.~{O}n
  certain {$6$}-manifolds,
  \href{https://doi.org/10.1007/BF01389738}{\textit{Invent. Math.}} \textbf{1}
  (1966), 355--374.

\bibitem{Wa67}
Wall C.T.C., Corrigendum to ``Classification problems in differential topology.
  {V}.~{O}n certain {$6$}-manifolds'',
  \href{https://doi.org/10.1007/BF01425407}{\textit{Invent. Math.}} \textbf{2}
  (1967), 306.

\bibitem{Wa22}
Wang X., On the classification of certain 1-connected 7-manifolds and related
  problems, \href{https://doi.org/10.1093/qmath/haab053}{\textit{Q.~J.~Math.}}
  \textbf{73} (2022), 711--727,
  \href{https://arxiv.org/abs/1810.08474}{arXiv:1810.08474}.

\bibitem{Wr97}
Wraith D., Exotic spheres with positive {R}icci curvature,
  \href{https://doi.org/10.4310/jdg/1214459846}{\textit{J.~Differential Geom.}}
  \textbf{45} (1997), 638--649.

\bibitem{Xu25}
Xu F., Free circle actions on certain simply connected 7-manifolds,
  \href{https://doi.org/10.1016/j.topol.2025.109548}{\textit{Topology Appl.}}
  \textbf{375} (2025), 109548, 26~pages,
  \href{https://arxiv.org/abs/2409.04938}{arXiv:2409.04938}.

\bibitem{Zu75}
Zhubr A.V., Classification of simply connected six-dimensional spin manifolds,
  \href{https://doi.org/10.1070/IM1975v009n04ABEH001498}{\textit{Math.
  USSR-Izv.}} \textbf{9} (1975), 793--812.

\bibitem{Zh00}
Zhubr A.V., Closed simply connected 6-manifolds: the proofs of classification
  theorems, \textit{St.~Petersburg Math.~J.} \textbf{12} (2001), 605--680.

\end{thebibliography}
\end{document}